\documentclass{amsart}
\usepackage{amsmath, amsfonts,amsthm,amssymb,amscd, verbatim,graphicx,color,multirow,booktabs, caption,tikz,tikz-cd, mathdots,bm}
\usepackage[margin=1.5cm]{geometry}
\usepackage{tikz-cd}
\newtheorem{theorem}{Theorem}[section]
\newtheorem{lemma}[theorem]{Lemma}
\newtheorem{fact}[theorem]{Fact}
\newtheorem{hypothesis}[theorem]{Hypothesis}
\newtheorem{remarks}[theorem]{Remarks}
\newtheorem{remark}[theorem]{Remark}
\newtheorem{corollary}[theorem]{Corollary}
\newcommand{\Alt}{\mathop{\mathrm{Alt}}}
\newcommand{\Aut}{\mathop{\mathrm{Aut}}}
\newcommand{\Sym}{\mathop{\mathrm{Sym}}}
\newcommand{\PGL}{\mathop{\mathrm{PGL}}}
\def\nor#1#2{{\bf N}_{{#1}}({{#2}})}
\def\FF#1#{{{\bf F}^*}{{({{#1}})}}}
\def\Zent#1#{{{\bf Z}}{{({{#1}})}}}
\begin{document}
\title[Boolean lattices]{Boolean lattices in finite alternating and symmetric groups} 
\author[A. Lucchini]{Andrea Lucchini}
\address{Andrea Lucchini, Dipartimento di Matematica \lq\lq Tullio Levi-Civita\rq\rq,\newline
 University of Padova, Via Trieste 53, 35121 Padova, Italy} 
\email{lucchini@math.unipd.it}
         
\author[M. Moscatiello]{Mariapia Moscatiello}
\address{Mariapia Moscatiello, Dipartimento di Matematica \lq\lq Tullio Levi-Civita\rq\rq,\newline
 University of Padova, Via Trieste 53, 35121 Padova, Italy} 
\email{mariapia.moscatiello@math.unipd.it}
\author[S. Palcoux]{Sebastien Palcoux}
\address{Sebastien Palcoux, Yau Mathematical Sciences Center, Tsinghua University, Beijing, China}
\email{sebastien.palcoux@gmail.com}
\author[P. Spiga]{Pablo Spiga}
\address{Pablo Spiga, Dipartimento di Matematica e Applicazioni, University of Milano-Bicocca, Via Cozzi 55, 20125 Milano, Italy} 
\email{pablo.spiga@unimib.it}
\begin{abstract}
Given a group $G$ and a subgroup $H$, we let $\mathcal{O}_G(H)$ denote the lattice of subgroups of $G$ containing $H$. 
This paper provides a classification of the subgroups $H$ of $G$ such that $\mathcal{O}_{G}(H)$ is Boolean of rank at least $3$, when $G$ is a finite alternating or symmetric group. Besides some sporadic examples and some twisted versions, there are two different types of such lattices. One type arises by taking stabilizers of  chains of regular partitions, and the other type arises by taking  stabilizers of chains of regular product structures. As an application, we prove in this case a conjecture on Boolean overgroup lattices, related to the dual Ore's theorem and to a problem of Kenneth Brown.
\end{abstract}

\maketitle

\section{Introduction}\label{sec:intro}

Given a finite group $G$ and a subgroup $H$ of $G$, we are interested in the set $$\mathcal{O}_G(H):=\{K\mid K\textrm{ subgroup of }G \textrm{ with }H\le K\}$$
of subgroups of $G$ containing $H$. Clearly, $\mathcal{O}_G(H)$ is a lattice under the operations of taking ``intersection'' and taking ``subgroup generated''; it is called the \emph{overgroup lattice}.  The problem of determining whether every finite lattice is isomorphic to some $\mathcal{O}_{G}(H)$ with $G$ finite arose originally in universal algebra with the work of P\'{a}lfy-Pudl\'{a}k \cite{PaPu}.

\O ystein Ore proved in 1938 that a finite group is cyclic if and only if its subgroup lattice is distributive \cite[Theorem~4]{or}, and he extended one way as follows: let $G$ be a finite group and $H$ a subgroup such that the overgroup lattice $\mathcal{O}_{G}(H)$ is distributive, then there is a coset $Hg$ generating $G$ \cite[Theorem 7]{or}. Eighty years later, the third author extended Ore's theorem to subfactor planar algebras \cite{pPJM,pQT} and applied it back to finite group theory as a dual version of Ore's theorem \cite{pJA} stating that if $\mathcal{O}_{G}(H)$ is distributive then there is an irreducible complex representation (irrep) $V$ of $G$ such that $G_{(V^H)} = H$, with $V^H$ the fixed point subspace and $G_{(X)}$ the point-wise stabilizer subgroup. An other way to prove this application (explored with Mamta Balodi \cite{bpJCTA}) is to show that the \emph{dual Euler totient} is nonzero. Let us explain what it means. Let $G$ be a finite group, its Euler totient $\varphi(G)$ is the number of elements $g$ such that $\langle g \rangle = G$. Then $\varphi(G)$ is nonzero if and only if $G$ is cyclic, and $\varphi(C_n) = \varphi(n)$, the usual Euler's totient function. For a subgroup $H \subset G$, the Euler totient $\varphi(H,G)$ is the number of cosets $Hg$ such that $\langle Hg \rangle = G$. Hall \cite{hal} reformulated it using the M\"obius function $\mu$ on the overgroup lattice $\mathcal{O}_{G}(H)$ as follows:  
$$ \varphi(H,G) = \sum_{K \in \mathcal{O}_{G}(H)} \mu(K,G) |K:H|.$$
In particular, $\varphi(H,G)$ is nonzero (if and) only if there is a coset $Hg$ generating $G$. Again that was extended to subfactor planar algebra \cite{pPAMS} and applied back as a dual version stating that for any subgroup $H \subset G$, if the \emph{dual Euler totient}
$$ \hat{\varphi}(H,G) := \sum_{K \in \mathcal{O}_{G}(H)} \mu(H,K) |G:K|,$$
is nonzero then there is an irrep $V$ such that $G_{(V^H)} = H$ (in particular, if $\hat{\varphi}(G):=\hat{\varphi}(1,G)$ is nonzero then $G$ is \emph{linearly primitive}, i.e. admits a faithful irrep). So the dual Ore's theorem appears as a natural consequence of \cite[Conjecture~1.5]{bpJCTA} stating that $\hat{\varphi}(H,G)$ is nonzero if $\mathcal{O}_{G}(H)$ is Boolean. More strongly, one asked \cite[page~58]{bpJCTA} whether the lower bound $\hat{\varphi}(H,G) \ge 2^{\ell}$ holds when $\mathcal{O}_{G}(H)$ is Boolean of rank $\ell+1$; if so, it is optimal because $\hat{\varphi}(S_1 \times S_2^{\ell}, S_2 \times S_3^{\ell}) = 2^{\ell}$. Now, this conjecture is a particular case of a relative version of a problem \textit{essentially} due to Kenneth S. Brown asking whether the M\"obius invariant of the bounded coset poset $P$ of a finite group (which is equal to the reduced Euler characteristic of the order complex of the proper part of $P$) is nonzero (\cite[page~760]{sw} and \cite[Question~4]{br}). In its relative version, the reduced Euler characteristic is given by
$$ \chi(H,G) = -\sum_{K \in \mathcal{O}_{G}(H)} \mu(K,G) |G:K|$$
but in the (rank $\ell$) Boolean case $\mu(H,K) = (-1)^{\ell} \mu(K,G)$, so $\chi(H,G) = \pm \hat{\varphi}(H,G)$, and the first is nonzero if and only if the second is so. We recalled in \cite[Example~4.21]{bpJCTA} that if $H$ is the Borel subgroup of a BN-pair structure (of rank $\ell$) on $G$, then $\mathcal{O}_{G}(H)$ is Boolean (of rank $\ell$), and  $\chi(H,G)$ is nonzero, and if moreover $G$ is a finite simple group of Lie type (over a finite field of characteristic $p$) then its absolute value $\hat{\varphi}(H,G)$ is the $p$-contribution in the order of $G$, which is at least $p^{\frac{1}{2}\ell(\ell+1)}$.   
A first step in an approach for this conjecture could be to prove the case where $G$ is a finite simple group, and for so, we can try to first classify the inclusions $H \subset G$ with $\mathcal{O}_{G}(H)$ Boolean and $G$ finite simple. But does the BN-pair structure cover everything at rank $\ge 3$, or large enough? John Shareshian answered no in an exchange on MathOverflow, by suggesting examples of any rank when $G$ is alternating, involving stabilizers of non-trivial regular partitions. This paper proves the existence of these examples for $G$ alternating (or symmetric),  but mainly proves that (besides some sporadic cases) there is just one other infinite family of examples arising from stabilizers of regular product structures. As a consequence, we can prove the above conjecture in this case (together with the lower bound).

We consider the case that $G$ is an almost simple group with socle an alternating group $\Alt(n)$, for some $n\in\mathbb{N}$. When $n\le 5$, nothing interesting happens: the largest Boolean lattice of the form $\mathcal{O}_G(H)$ has rank at most $1$. Moreover, since the case $n=6$ is rather special, we deal with this case separately. When $G=\Alt(6)$, the largest Boolean lattice has rank $2$ and it is of the form $(D_4,\Sym(4),\Sym(4))$ or $(D_{5},\Alt(5),\Alt(5))$. When $G=\PGL_2(9)$, $G=M_{10}$ or $\mathrm{P}\Gamma\mathrm{L}_2(9)$, the largest Boolean lattice has rank $1$. When $G=\Sym(6)\cong \mathrm{P}\Sigma\mathrm{L}_2(9)$, the largest Boolean lattice has rank $2$ and it is of the form $(D_4\times C_2,2.\Sym(4),2.\Sym(4))$ or $(C_5\rtimes C_4,\Sym(5),\Sym(5))$.

For the rest of the argument, we may suppose $n\ne 6$ and hence for the rest of this paper we assume $G=\Alt(\Omega)$ or $G=\Sym(\Omega)$, where $\Omega$ is a finite set.

\begin{theorem}\label{thrm:main}Let $\Omega$ be a finite set, let $G$ be $\Alt(\Omega)$ or $\Sym(\Omega)$, let $H$ be a subgroup of $G$ and suppose that the lattice $\mathcal{O}_G(H)=\{K\mid H\le K\le G\}$ is Boolean of rank $\ell\ge 3$. Let $G_1,\ldots,G_\ell$ be the maximal elements of $\mathcal{O}_G(H)$. Then one of the following holds:
\begin{enumerate}
\item\label{main1} For every $i\in \{1,\ldots,\ell\}$, there exists a non-trivial regular partition $\Sigma_i$ with $G_i=\nor G{\Sigma_i}$; moreover, relabeling the indexed set $\{1,\ldots,\ell\}$ if necessary, $\Sigma_1< \cdots <\Sigma_\ell$.
\item\label{main2} $G=\Sym(\Omega)$. Relabeling the indexed set $\{1,\ldots,\ell\}$ if necessary, $G_\ell=\Alt(\Omega)$, for every $i\in \{1,\ldots,\ell-1\}$, there exists a non-trivial regular partition $\Sigma_i$ with $G_i=\nor G{\Sigma_i}$; moreover, relabeling the indexed set $\{1,\ldots,\ell-1\}$ if necessary, $\Sigma_1< \cdots <\Sigma_{\ell-1}$.
\item\label{main3} $|\Omega|$ is odd. For every $i\in \{1,\ldots,\ell\}$, there exists a non-trivial regular product structure $\mathcal{F}_i$ with $G_i=\nor G{\mathcal{F}_i}$; moreover, relabeling the indexed set $\{1,\ldots,\ell\}$ if necessary, $\mathcal{F}_1< \cdots < \mathcal{F}_\ell$.
\item\label{main4} $|\Omega|$ is an odd and $G=\Sym(\Omega)$. Relabeling the indexed set $\{1,\ldots,\ell\}$ if necessary, $G_\ell=\Alt(\Omega)$, for every $i\in \{1,\ldots,\ell-1\}$, there exists a non-trivial regular product structure $\mathcal{F}_i$ with $G_i=\nor G{\mathcal{F}_i}$; moreover, relabeling the indexed set $\{1,\ldots,\ell-1\}$ if necessary, $\mathcal{F}_1< \cdots < \mathcal{F}_{\ell-1}$.
\item\label{main5} $|\Omega|$ is an odd prime power. Relabeling the indexed set $\{1,\ldots,\ell\}$ if necessary, $G_\ell$ is maximal subgroup of O'Nan-Scott type HA, for every $i\in \{1,\ldots,\ell-1\}$, there exists a non-trivial regular product structure $\mathcal{F}_i$ with $G_i=\nor G{\mathcal{F}_i}$; moreover, relabeling the indexed set $\{1,\ldots,\ell-1\}$ if necessary, $\mathcal{F}_1< \cdots < \mathcal{F}_{\ell-1}$.
\item\label{main6} $|\Omega|$ is odd prime power and $G=\Sym(\Omega)$. Relabeling the indexed set $\{1,\ldots,\ell\}$ if necessary, $G_\ell=\Alt(\Omega)$ and $G_{\ell-1}$ is a maximal subgroup of O'Nan-Scott type HA, for every $i\in \{1,\ldots,\ell-2\}$, there exists a non-trivial regular product structure $\mathcal{F}_i$ with $G_i=\nor G{\mathcal{F}_i}$; moreover, relabeling the indexed set $\{1,\ldots,\ell-2\}$ if necessary, $\mathcal{F}_1< \cdots < \mathcal{F}_{\ell-2}$.

\item\label{main7} $\ell=3$, $G=\Sym(\Omega)$ and, relabeling the indexed set $\{1,2,3\}$ if necessary,  $G_1$ is the stabilizer of a subset $\Gamma$ of $\Omega$ with $1\le |\Gamma|<|\Omega|/2$, $G_2$ is the stabilizer of a non-trivial regular partition $\Sigma$ with $\Gamma\in \Sigma$ and $G_3=\Alt(\Omega)$; 
\item\label{main8} $\ell=3$, $G=\Sym(\Omega)$ and, relabeling the indexed set $\{1,2,3\}$ if necessary,  $G_1$ is the stabilizer of a subset $\Gamma$ of $\Omega$ with $|\Gamma|=1$, $G_2\cong \mathrm{PGL}_2(p)$ for some prime number $p$, $|\Omega|=p+1$ and $G_3=\Alt(\Omega)$; 
\item $\ell=3$, $G=\Alt(\Omega)$, $|\Omega|=8$ and the Boolean lattice $\mathcal{O}_G(H)$ is in Figure~$\ref{fig1}$. 
\item\label{main9} $\ell=3$, $G=\Alt(\Omega)$, $|\Omega|=24$, and, relabeling the indexed set $\{1,2,3\}$ if necessary,  $G_1$ is the stabilizer of a subset $\Gamma$ of $\Omega$ with $|\Gamma|=1$, $G_2\cong G_3\cong M_{24}$.
\end{enumerate}
\end{theorem}

In Section~\ref{sec:construction}, we show that the cases in Theorem~\ref{thrm:main}~\eqref{main1} and~\eqref{main2} do occur for arbitrary values of $\ell$.  In Section~\ref{sec:construction2}, we show that there exist Boolean lattices of arbitary large rank whose maximal elements are stabilizers of regular product structures. 

Finally, Section~\ref{sec:brown} is dedicated to the proof of the following theorem where (\ref{bound3}) is a consequence of Theorem \ref{thrm:main}, and where the proof for (\ref{bound4}) was already mentioned above. 

\begin{theorem} \label{thrm:bound}
Let $G$ be a finite group and $H$ a subgroup such that the overgroup lattice $\mathcal{O}_{G}(H)$ is Boolean of rank $\ell$. Then the lower bound on the dual Euler totient $\hat{\varphi}(H,G) \ge 2^{\ell-1}$ holds in each of the following cases:  
\begin{enumerate}
\item\label{bound1}  $\ell \le 3$,
\item\label{bound2'} $\mathcal{O}_{G}(H)$ group-complemented,
\item\label{bound2}  $G$ solvable,
\item\label{bound3}  $G$ alternating or symmetric,
\item\label{bound4}  $G$ of Lie type and $H$ a Borel subgroup.
\end{enumerate}
As a consequence, the reduced Euler characteristic $\chi(H,G)$ is nonzero, i.e. it is a positive answer to the relative Brown's problem in these cases.
\end{theorem}

\tableofcontents

\section{Notation, Terminology and basic facts}\label{sec:basic}
Since we need fundamental results from the work of Aschbacher~\cite{1,2}, we follow the notation and the terminology therein. We let $G$ be the finite alternating group $\Alt(\Omega)$ or the finite symmetric group $\Sym(\Omega)$, where $\Omega$ is a finite set of cardinality $n\in\mathbb{N}$. Given a subgroup $H$ of $G$, we write $$\mathcal{O}_G(H):=\{K\mid H\le K\le G\}$$ for the set of subgroups of $G$ containing $H$. We let $$\mathcal{O}_G(H)':=\mathcal{O}_G(H)\setminus\{H,G\},$$ that is, $\mathcal{O}_G(H)'$ consists of the lattice $\mathcal{O}_G(H)$ with its minimum and its maximum elements removed. (Given a group $X$, we denote by ${\bf F}^*(X)$ the {\em generalized Fitting} subgroup of $X$. Observe that, when $X$ is a primitive subgroup of $\Sym(\Omega)$, ${\bf F}^*(X)$ coincides with the {\em socle} of $X$.)
We write
$$\mathcal{O}_G(H)'':=\{M\in\mathcal{O}_G(H)\mid {\bf F}^*(G)\nleq M\}$$
and we denote by $$\mathcal{M}_G(H)\quad\textrm{ the set of maximal members of }\mathcal{O}_G(H)''.$$ This notation is due to Aschbacher~\cite{1,2} and it is rather important for the results we need to recall from his work. Therefore, we start with familiarizing with this new terminology.
\begin{itemize}
\item When $G=\Alt(\Omega)$, ${\bf F}^*(G)=G$ and hence $\mathcal{O}_G(H)''$ is simply $\mathcal{O}_G(H)$ with its maximum element $G=\Alt(\Omega)$ removed. Therefore $\mathcal{M}_G(H)$ consists of the maximal subgroups of $G=\Alt(\Omega)$ containing $H$. 
\item When $G=\Sym(\Omega)$ and $\Alt(\Omega)\nleq H$,  $\mathcal{O}_G(H)''$ is obtained from $\mathcal{O}_G(H)$ by removing $G=\Sym(\Omega)$ only, because if $M\in \mathcal{O}_G(H)$ and $\Alt(\Omega)={\bf F}^*(G)\le M$, then $\Sym(\Omega)=H{\bf F}^*(G)\le M$ and $M=\Sym(\Omega)$. Therefore, also in this case $\mathcal{M}_G(H)$ consists simply of the maximal subgroups of $G=\Sym(\Omega)$ containing $H$. 
\item When $G=\Sym(\Omega)$ and $H\le \Alt(\Omega)$, $\mathcal{O}_G(H)''$  is obtained from $\mathcal{O}_G(H)$ by removing $\Sym(\Omega)$ and $\Alt(\Omega)$. Therefore $\mathcal{M}_G(H)$ consists of two types of subgroups: the maximal subgroups of $G=\Sym(\Omega)$ containing $H$ and the maximal subgroups $M$ of $\Alt(\Omega)$ containing $H$ and that are not contained in any other maximal subgroup of $\Sym(\Omega)$ other then $\Alt(\Omega)$. For instance, when $H:=M_{12}$ in its transitive action of degree $12$, we have $H\le \Alt(12)$, $\mathcal{O}_{\Sym(\Omega)}(M_{12})=\{M_{12},\Alt(12),\Sym(12)\}$, $\mathcal{O}_{\Sym(12)}(H)'=\{\Alt(12)\}$, $\mathcal{O}_{\Sym(12)}(M_{12})''=\{M_{12}\}$ and $\mathcal{M}_{\Sym(12)}(M_{12})=\{M_{12}\}$.
\end{itemize}

Some of the material that follows can be traced back to~\cite{1,2} or~\cite{16,18}. However, we prefer to repeat it here because it helps to set some more notation and terminology. Using the action of $\Sym(\Omega)$ on the domain $\Omega$, we can divide the subgroups $X$ of $\Sym(\Omega)$ into three classes:
\begin{description}
\item[Intransitive]$X$ is  {\em intransitive} on $\Omega$,
\item[Imprimitive]$X$ is {\em imprimitive} on $\Omega$, that is, $X$ is transitive on $\Omega$ but it is not primitive on $\Omega$,
\item[Primitive]$X$ is {\em primitive} on $\Omega$.
\end{description}
In particular, every maximal subgroup $M$ of $G$ can be referred to as intransitive, imprimitive or primitive, according to the division above.

In what follows we need detailed information on the overgroups of a primitive subgroup of $G$. This information was obtained independently by Aschbacher~\cite{1,2} and Liebeck, Praeger and Saxl~\cite{16,18}. Both investigations are important in what follows.

\subsection{Intransitive subgroups}A maximal subgroup $M$ of $G$  is intransitive if and only if $M$ is the stabilizer in $G$ of a subset $\Gamma$ of $\Omega$ with $1\le |\Gamma|<|\Omega|/2$ (see for example~\cite{16}), that is,
$$M=G\cap (\Sym(\Gamma)\times \Sym(\Omega\setminus \Gamma)).$$ 
Following~\cite{1,2}, we let $\nor G\Gamma$ denote the setwise stabilizer of $\Gamma$ in $G$, that is,
$$\nor G\Gamma:=\{g\in G\mid \gamma^g\in\Gamma, \forall\gamma\in \Gamma\}.$$
(More generally, given a subgroup $H$ of $G$, we let $\nor H \Gamma=\nor G\Gamma\cap H$ denote the setwise stabilizer of $\Gamma$ in $H$.) 
The case $|\Gamma|=|\Omega|/2$ is special because the action of $\Sym(\Omega)$ on the subsets of $\Omega$ of cardinality $|\Omega|/2$ is imprimitive. Indeed, $\{\{\Gamma,\Omega\setminus \Gamma\}\mid \Gamma\subseteq \Omega, |\Gamma|=|\Omega|/2\}$ is a system of imprimitivity for the action of $\Sym(\Omega)$ on the subsets of $\Omega$ of cardinality $|\Omega|/2$. 
Summing up, we have the following fact. 

\begin{fact}\label{fact1}
{\rm 
Let $\Gamma$ be a subset of $\Omega$ with $1\le |\Gamma|<|\Omega|/2$. Then, the intransitive  subgroup $\nor G \Gamma$ of $G$ is a  maximal subgroup of $G$. Moreover, every maximal subgroup of $G$ which is intransitive is of this form.
}
\end{fact}
 
\subsection{Regular partitions and imprimitive subgroups}The collection of all partitions of $\Omega$ is a poset: given two partitions $\Sigma_1$ and $\Sigma_2$ of $\Omega$, we say that $\Sigma_1\le \Sigma_2$ if $\Sigma_2$ is a {\em refinement} of $\Sigma_1$, that is, every element in $\Sigma_1$ is a union of elements in $\Sigma_2$. For instance, when $\Omega:=\{1,2,3,4\}$, $\Sigma_1:=\{\{1,3,4\},\{2\}\}$ and $\Sigma_2:=\{\{1\},\{2\},\{3,4\}\}$, we have $\Sigma_1\le \Sigma_2$.

A partition $\Sigma$ of $\Omega$ is said to be {\em regular} or {\em uniform} if all parts in $\Sigma$ have the same cardinality. Following~\cite{1,2}, we say that the partition $\Sigma$ is an $(a,b)$-{\em regular partition} if $\Sigma$ consists of $b$ parts each having cardinality $a$. In particular, $n=|\Omega|=ab$.
  
A partition $\Sigma$ of $\Omega$ is said to be {\em trivial} if $\Sigma$ equals the universal relation $\Sigma=\{\Omega\}$ or if $\Sigma$ equals the equality relation $\Sigma=\{\{\omega\}\mid \omega\in \Omega\}$.  

We let $$\nor G {\Sigma}:=\{g\in G\mid \Gamma^g\in \Sigma,\forall \Gamma\in \Sigma\}$$
denote the stabilizer in $G$ of the partition $\Sigma$. Moreover, when $H$ is a subgroup of $G$, we write $\nor H\Sigma:=\nor G \Gamma\cap H$.

Let $M$ be a maximal subgroup of $G$. If $M$  is imprimitive, then $M$ is the stabilizer in $G$ of a non-trivial regular partition. Therefore, there exists an $(a,b)$-regular partition $\Sigma$ with $a,b\ge 2$ and with $M=\nor G \Sigma$. From~\cite{16,18}, we see that when $G=\Sym(\Omega)$ the converse is also true. That is, for every non-trivial $(a,b)$-regular partition $\Sigma$, the subgroup $\nor G\Sigma$ is a maximal subgroup of $\Sym(\Omega)$. When $G=\Alt(\Omega)$, the converse is not quite true in general. We summarize what we need in the following fact.

\begin{fact}\label{fact2}
{\rm 
Let $\Sigma$ be a non-trivial regular partition of $\Omega$. Except when $G=\Alt(\Omega)$, $|\Omega|=8$ and $\Sigma$ is a $(2,4)$-regular partition, the imprimitive  subgroup $\nor G \Sigma$ of $G$ is a  maximal subgroup of $G$.}
\end{fact}
The case $G=\Alt(\Omega)$, $|\Omega|=8$ and $\Sigma$ is a $(2,4)$-regular partition is a genuine exception here. Indeed,  $\nor G\Sigma<\mathrm{AGL}_3(2)<\Alt(\Omega)$, where $\mathrm{AGL}_3(2)$ is the affine general linear group of degree $2^3=8$.  (This was already observed in~\cite{16}.) The case $G=\Alt(\Omega)$ and $n=8$  is combinatorially very interesting: the largest Boolean lattice in $\Alt(8)$ has rank $3$ and it is drawn in Figure~\ref{fig1}.
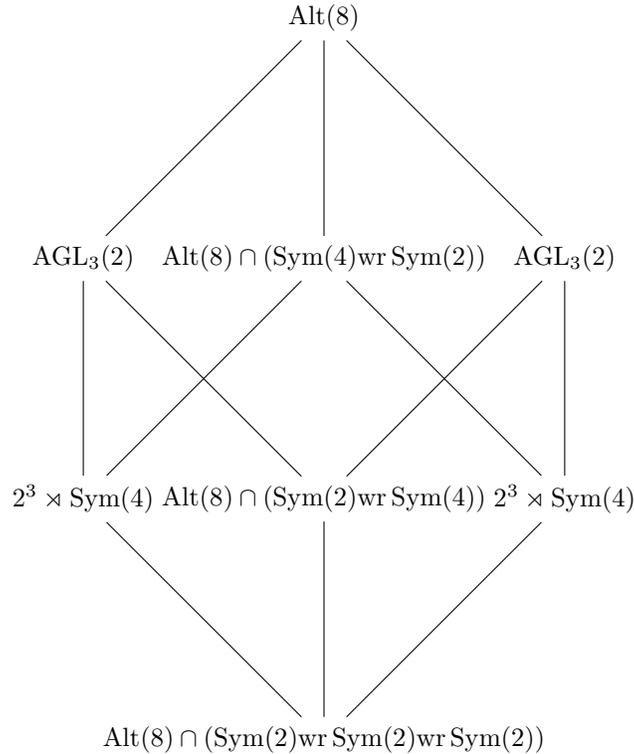
\begin{figure}
\begin{tikzpicture}[node distance=3.2cm]
\node(A0){$\Alt(8)$};
\node(A1)[below of=A0]{$\Alt(8)\cap (\Sym(4)\mathrm{wr}\Sym(2))$};
\node(A2)[left of=A1]{$\mathrm{AGL}_3(2)$};
\node(A3)[right of=A1]{$\mathrm{AGL}_3(2)$};
\node(A4)[below of=A1]{$\Alt(8)\cap(\Sym(2)\mathrm{wr}\Sym(4))$};
\node(A5)[below of=A2]{$2^3\rtimes\Sym(4)$};
\node(A6)[below of=A3]{$2^3\rtimes\Sym(4)$};
\node(A7)[below of=A4]{$\Alt(8)\cap (\Sym(2)\mathrm{wr}\Sym(2)\mathrm{wr}\Sym(2))$};
\draw(A0)--(A1);
\draw(A0)--(A2);
\draw(A0)--(A3);
\draw(A1)--(A5);
\draw(A1)--(A6);
\draw(A2)--(A5);
\draw(A2)--(A4);
\draw(A3)--(A6);
\draw(A3)--(A4);
\draw(A7)--(A6);
\draw(A7)--(A5);
\draw(A7)--(A4);
\end{tikzpicture}
\caption{The Boolean lattice of largest cardinality in $\Alt(8)$} \label{fig1}
\end{figure}

\subsection{Regular product structures and primitive subgroups}The modern key for analysing a finite primitive permutation group $L$ is to
study the \textit{socle} $N$ of $L$, that is, the subgroup generated
by the minimal normal subgroups of $L$. The socle of an arbitrary
finite group is isomorphic to the non-trivial direct product of simple
groups; moreover, for finite primitive groups these simple groups are
pairwise isomorphic. The O'Nan-Scott theorem describes in details the
embedding of $N$ in $L$ and collects some useful information about the
action of $L$. In~\cite[Theorem]{LPSLPS} five types of primitive groups
are defined (depending on the group- and action-structure of the
socle), namely HA (\emph{Affine}), AS (\emph{Almost Simple}), SD
(\emph{Simple Diagonal}), PA (\emph{Product Action}) and TW
(\emph{Twisted Wreath}), and it is shown that  every primitive group
belongs to exactly one of these types. We remark that in~\cite{C3}
this division into types is refined, namely the PA type
in~\cite{LPSLPS} is partitioned in four parts, which are called HS (\emph{Holomorphic simple}), HC (\emph{Holomorphic compound}), CD
(\emph{Compound Diagonal}) and PA.  For what follows it is convenient to use this division into eight types of the finite primitive primitive groups.

It follows from the results in~\cite{16,18} that, if $M$ is a maximal subgroup of $G$ and $M$ is primitive, then $M$ has O'Nan-Scott type HA, SD, PA or AS.

Since an overgroup of a primitive group is still primitive, the analogue of Facts~\ref{fact1} and~\ref{fact2} is obvious.
\begin{fact}\label{fact3}
{\rm  A primitive subgroup $M$ of $G$ is maximal if and only if $M$ is maximal among the primitive subgroups of $G$.}
\end{fact}
We recall the definition of a {\em regular product structure}  on $\Omega$ from~\cite[Section~$2$]{2}. Let $m$ and $k$ be integers with $m \ge 5$ and $k \ge 2$. There are two natural ways to do that. First, a regular $( m , k )$-product structure on $\Omega$ is a bijection $f : \Omega \to \Gamma^{I}$, where $I := \{ 1 , \ldots , k \}$ and $\Gamma$ is an $m$-set. The function $f$ consists of a family of functions $( f_i : \Omega \to \Gamma \mid i \in I )$ where  $f ( \omega ) = ( f_1 ( \omega ), \ldots , f_k ( \omega ))$, for each $\omega \in \Omega$. There is a more intrinsict way to define it. A product structure is a set $\mathcal{F} := \{\Omega_i \mid i \in I \}$ of partitions $\Omega_i$ of $\Omega$ into $m$ blocks of size $m^{k - 1}$, such that, for each pair of distinct points $\omega,\omega'\in \Omega$, we have $\mathcal{F} ( \omega )\ne\mathcal{F} ( \omega' )$, where
$\mathcal{F} ( \omega ):=\{[ \omega ]_i \mid i \in I\}$ consists of the blocks defined by $\omega \in [ \omega ]_i$ and $[\omega]_i \in \Omega_i$ (here $[\omega]_i$ denotes the block of $\Omega_i$ containing the point $\omega$).
Clearly the two definitions are equivalent. Indeed, given a function $f:\Omega\to \Gamma^{I}$, we let  $\mathcal{F} ( f )$ be the set of partitions of $\Omega$ defined by $f$, where the $i^{\mathrm{th}}$ partition $\Omega_i := \{ f_i^{-1} ( \gamma ) \mid  \gamma\in \Gamma \}$ consists of the the fibers
of $f_i$. 
The product structure $\mathcal{F}$ can also be regarded as a chamber system in the sense of Tits~\cite{boobs}.

Following~\cite{1}, we let $\nor G{\mathcal{F}}$ denote the stabilizer of a regular $(m,k)$-product structure $\mathcal{F}=\{\Omega_1,\ldots,\Omega_k\}$ in $G$, that is,
$$\nor G{\mathcal{F}}:=\{g\in G\mid \Omega_i^g\in\mathcal{F}, \forall i\in \{1,\ldots,k\}\}.$$
(More generally, given a subgroup $H$ of $G$, we let $\nor H {\mathcal{F}}:=\nor G{\mathcal{F}}\cap H$ denote the  stabilizer of $\Gamma$ in $H$.) Clearly,
$$\nor G{\mathcal{F}}\cong G\cap \left(\Sym(m)\mathrm{wr}\Sym(k)\right),$$
where $\Sym(m)\mathrm{wr}\Sym(k)$ is endowed of its primitive product action of degree $m^k$. Moreover, $\nor {\Sym(\Omega)}{\mathcal{F}}$ is a typical primitive maximal subgroup of $\Sym(\Omega)$ of PA type according to the O'Nan-Scott theorem.

Let $\bm{\mathcal{F}}(\Omega)$ be the
set of all regular product structures on $\Omega$. The set $\bm{\mathcal{F}}(\Omega)$ is endowed of a natural partial order. Let $\mathcal{F} := \{\Omega_i \mid i \in I\}$ and  $\tilde{\mathcal{F}} := \{ \tilde{\Omega}_j \mid j \in\tilde{I}\}$ be regular $(m,k)$- and $(\tilde{m},\tilde{k})$-product structures on $\Omega$, respectively.
Set $I := \{1, \ldots ,k\}$ and $\tilde{I} := \{1, \ldots ,\tilde{k}\}$, and define $\mathcal{F} \le \tilde{\mathcal{F}}$ if there exists a positive integer $s$ with $\tilde{k} =ks$,
and a regular $(s,k)$-partition $\Sigma = \{\sigma_i \mid i \in I\}$ of $\tilde{I}$, such that for each $i \in I$ and each $j \in \sigma_i$, $\tilde{\Omega}_j \le \Omega_i$, that is, the partition $\Omega_i$ is a refinement of the partition $\tilde{\Omega}_j$. From~\cite[(5.1)]{1}, the relation $\le$ is a partial order on $\bm{\mathcal{F}}(\Omega)$.

We conclude this preliminary observations on regular product structures by recalling~\cite[(5.10)]{2}.
\begin{lemma}\label{l:aschbacher}
Let $M=\nor {\Sym(\Omega)}{\mathcal{F}}$ be the stabilizer in $\Sym(\Omega)$ of a regular $(m,k)$-product structure on $\Omega$ and let $K$ be the kernel of the action of $M$ on $\mathcal{F}$. Then
\begin{enumerate}
\item\label{asch1}$K\le \Alt(\Omega)$ if and only if $m$ is even;
\item\label{asch2}$M\le \Alt(\Omega)$ if and only if $m$ is even and either $k>2$, or $k=2$ and $m\equiv 0\pmod 4$;
\item\label{asch3}if $k=2$ and $m\equiv 2\pmod 4$, then $M\cap \Alt(\Omega)=K$, so $M\cap \Alt(\Omega)$ is not primitive on $\Omega$ (and hence $M\cap \Alt(\Omega)$ is not a maximal subgroup of $\Alt(\Omega)$). Otherwise $M\cap \Alt(\Omega)$ induces $\Sym(\mathcal{F})$ on $\mathcal{F}$.
\end{enumerate}
\end{lemma}

\medskip

\subsection{Preliminary lemmas}
A lattice $\mathcal{L}$ is said to be {\em Boolean} if $\mathcal{L}$ is isomorphic to the lattice of subsets of a set $X$, that is, $\mathcal{L}\cong \mathcal{P}(X)$, where $\mathcal{P}(X):=\{Y\mid Y\subseteq X\}$. We also say that $|X|$ is the {\em rank} of the Boolean lattice $\mathcal{L}$.

\begin{lemma}\label{l:-01}Let $X$ be a subgroup of $Y$. If $\mathcal{O}_Y(X)$ is Boolean of rank $\ell$, then every maximal chain from $X$ to $Y$ has length $\ell+1$. In particular, if $|Y:X|$ is divisible by at most $\ell$ primes, then $\mathcal{O}_Y(X)$ is not Boolean of rank $\ell$.
\end{lemma}
\begin{proof}
This is clear.
\end{proof}

\begin{lemma}\label{l:0}Let $H$ be a subgroup of $G$ with $\mathcal{O}_G(H)$ Boolean. If every maximal element in $\mathcal{O}_G(H)$ is transitive, then either $H$ is transitive or $\mathcal{O}_G(H)$ contains the stabilizer of a $(|\Omega|/2,2)$-regular partition.
\end{lemma}
\begin{proof}
Suppose that $H$ is intransitive and let $\Gamma$ be an orbit of $H$ of smallest possible cardinality. Assume $1\le |\Gamma|<|\Omega|/2$. Then $M:=G\cap (\Sym(\Gamma)\times \Sym(\Omega\setminus \Gamma))$ is a maximal element of $\mathcal{O}_G(H)$ and $M$ is intransitive, which is a contradiction. This shows that $H$ has two orbits on $\Omega$ both having cardinality $|\Omega|/2$. In particular, $M:=\nor {G}{\{\Gamma,\Omega\setminus\Gamma\}}$ is a member of $\mathcal{O}_G(H)$. 
\end{proof}

\begin{lemma}\label{l:1}Let $H$ be a subgroup of $G$ with $\mathcal{O}_G(H)$ Boolean. If every maximal element in $\mathcal{O}_G(H)$ is primitive, then either $H$ is primitive, or $G=\Alt(\Omega)$, $|\Omega|=8$, $H=\nor G\Sigma$ for some $(2,4)$-regular partition $\Sigma$ and $\mathcal{O}_G(H)$ has rank $2$.
\end{lemma}
\begin{proof}
From Lemma~\ref{l:0}, $H$ is transitive. Suppose that $H$ is imprimitive and let $\Sigma$ be a non-trivial regular partition with $H\le \nor G\Sigma$. If $\nor G\Sigma$ is a maximal subgroup of $G$, we obtain a contradiction. Thus $\nor G\Sigma$ is not maximal in $G$. This implies $G=\Alt(\Omega)$, $|\Omega|=8$, $\Sigma$ is a $(2,4)$-regular partition and $\mathcal{O}_G(H)$ has rank $2$: see Fact~\ref{fact2} and Figure~\ref{fig1}. 
\end{proof}

Lemma~\ref{l:2} is needed in Remark~\ref{rem:1} and Lemma~\ref{l:4} is needed in Theorem~\ref{thrm:AS2}.
\begin{lemma}\label{l:2}
Let $\Omega$ be the set of all pairs from a finite set $\Delta$. Then, in the permutation representation of $\Sym(\Delta)$ on $\Omega$, $\Sym(\Delta)\leq \Alt(\Omega)$ if and only if $|\Delta|$ is even. 
\end{lemma}
\begin{proof}
It is an easy computation to see that, if $g$ is a transposition of $\Sym(\Delta)$ (for its action on $\Delta$), then $g$ is an even permutation in its action on $\Omega$ if and only if $|\Delta|$ is even. Therefore, the proof follows.
\end{proof}

\begin{lemma}\label{l:4}
Let $H$ be a transitive permutation group on $\Omega$, let $\omega\in\Omega$ and let $H_\omega$ be the stabilizer of the point $\omega$ in $H$. Then $\{\omega'\in\Omega\mid \omega'^g=\omega',\forall g\in H_\omega\}$ is a block of imprimitivity for $H$. In particular, if $H$ is primitive, then either $H_\omega=1$, or $\omega$ is the only point fixed by $H_\omega$. 
\end{lemma}
\begin{proof}
This is an exercise, see~\cite[Exercise~$1.6.5$, page~$19$]{DM}.
\end{proof}


\section{Results for almost simple groups}\label{sec:AS}
In this section we collect some results from~\cite{1,2} on primitive groups. Our ultimate goal is deducing some structural results on Boolean lattices $\mathcal{O}_G(H)$, when $H$ is an almost simple primitive group 

We start with a rather technical result of Aschbacher on the overgroups of a primitive group which is {\em product indecomposable} and not {\em octal} . We prefer to give only a broad description of these concepts here and we refer the interested reader to~\cite{1,2}. These deep results have already played an important role in algebraic combinatorics; for instance, they are the key results for proving that most primitive groups are automorphism groups of edge-transitive hypergraphs~\cite{mine}.

 A primitive group $H\le G$ is said to be {\em product decomposable} if the domain $\Omega$ admits the structure of a Cartesian product (that is, $\Omega\cong \Delta^\ell$, for some finite set $\Delta$ and for some $\ell\in\mathbb{N}$ with $\ell\ge 1$) and the group $H$ acts on $\Omega$ preserving this Cartesian product structure. We are allowing $\ell=1$ here, to include the case that $H$ is almost simple. Moreover, for each component $L$ of the socle of $H$ one of the following holds:
\begin{description}
\item[(i)]$L\cong \Alt(6)$ and $|\Delta|=6^2$,
\item[(ii)]$L\cong M_{12}$ and $|\Delta|=12^2$,
\item[(iii)]$L\cong\mathrm{Sp}_4(q)$ for some $q>2$ even and $|\Delta|=(q^2(q^2-1)/2)^2$.
\end{description}
We also refer to~\cite{PS} for a recent thorough investigation on permutation groups admitting Cartesian decompositions, where each of these peculiar examples are thoroughly investigated.

Following~\cite{1,2}, a primitive group $H$ is said to be {\em octal} if each component $L$ of the socle of $H$  is isomorphic to $\mathrm{PSL}_3(2)\cong \mathrm{PSL}_2(7)$, the orbits of $L$ have order $8$ and the action of $L$ on each of its orbits is primitive. For future reference, we report here that a simple computation reveals that, when $H=\mathrm{PSL}_3(2)$  is octal,  $\mathcal{O}_{\Alt(8)}(H)$ is Boolean of rank $2$, whereas $\mathcal{O}_{\Sym(8)}(H)$ is a lattice of size $6$.

\begin{theorem}\label{thrm:A}{{\cite[Theorem~A]{2}}}
Let $\Omega$ be a finite set of cardinality $n$ and let $H$ be an almost simple primitive subgroup of $\Sym(\Omega)$ which is product indecomposable and not octal. Then all members of $\mathcal{O}_{\Sym(\Omega)}(H)$ are almost simple, product indecomposable, and not octal, and setting $U:={\bf F}^*(H)$, one of the following holds:
\begin{enumerate}
\item\label{thrmAenu1}$|\mathcal{M}_{\Sym(\Omega)}(H)|=1$.
\item\label{thrmAenu2}$U=H$, $|\mathcal{M}_{\Sym(\Omega)}(H)|=3$, $\Aut(U)\cong \nor {\Sym(\Omega)}U\in\mathcal{M}_{\Sym(\Omega)}(U)$, $\nor {\Sym(\Omega)}U$ is transitive on $\mathcal{M}_{\Sym(\Omega)}(H)\setminus\{\nor {\Sym(\Omega)}{U}\}$ and $U$ is maximal in $V$, where $K\in \mathcal{M}_{\Sym(\Omega)}(H)\setminus \{\nor {\Sym(\Omega)}{U}\}$ and $V={\bf F}^*(K)$. Further $(U,V,n)$ is one of the following:
\begin{description}
\item[(a)]$(HS,\Alt(m),15400)$, where $m=176$ and $n={m\choose 2}$.
\item[(b)]$(\mathrm{G}_2(3), \Omega_7(3),3159)$.
\item[(c)]$(\mathrm{PSL}_2(q),M_n,n)$, where $q\in \{11,23\}$, $n=q+1$ and $M_n$ is the Mathieu group of degree $n$.
\item[(d)]$(\mathrm{PSL}_2(17),\mathrm{Sp}_8(2),136)$.
\end{description}
\item\label{thrmAenu3}$U\cong\mathrm{PSL}_3(4)$, $n=280$, $|\mathcal{M}_{\Sym(\Omega)}(U)|=4$, $\Aut(U)\cong \nor {\Sym(\Omega)}U\in \mathcal{M}_{\Sym(\Omega)}(U)$, $\nor {\Sym(\Omega)}{U}$ is transitive on $\mathcal{M}_{\Sym(\Omega)}(U)\setminus \{\nor {\Sym(\Omega)}U\}$ and $K\in \mathcal{M}_{\Sym(\Omega)}(H)\setminus \{\nor {\Sym(\Omega)}U\}$ is isomorphic to $\Aut(\mathrm{PSU}_4(3))$.
\item\label{thrmAenu4}$U\cong \mathrm{Sz}(q)$, $q=2^k$, $n=q^2(q^2+1)/2$, $\mathcal{M}_{\Sym(\Omega)}(U)=\{K_1,K_2\}$ where $K_i=\nor {\Sym(\Omega)}{V_i}\cong \Aut(V_i)$, $V_1\cong \Alt(q^2+1)$, $V_2\cong \mathrm{Sp}_{4k}(2)$ and $\nor {\Sym(\Omega)}U\cong \Aut(U)$ is maximal in $V_1$.
\item\label{thrmAenu5}$H\cong\mathrm{PSL}_2(11)$, $n=55$, $\mathrm{PGL}_2(11)\cong \nor {\Sym(\Omega)}H$ and $\mathcal{M}_{\Sym(\Omega)}(H)=\{\nor {\Sym(\Omega)}{H},K,K^t\}$, $t\in \nor {\Sym(\Omega)}{H}\setminus H$, where $K\cong \Sym(11)$ and $\mathcal{O}_K(H)=\{H<L<V<K\}$, with $L\cong M_{11}$ and $V\cong \Alt(11)$.
\end{enumerate}
\end{theorem}
\begin{remarks}\label{rem:1}{\rm
\begin{enumerate}
\item In Case~\eqref{thrmAenu1}, since $\mathcal{M}_G(H)$ contains only one element, we deduce that the lattice $\mathcal{O}_{\Sym(\Omega)}(H)$ is not Boolean, unless it has rank $1$.
\item In Case~\eqref{thrmAenu2}~(a), all elements in $\mathcal{M}_{\Sym(\Omega)}(H)$ are maximal subgroups of $\Alt(\Omega)$. This is because the permutation representations of $\Aut(HS)=HS.2$ and of $\Sym(m)$ of degree ${m\choose 2}$ are the natural permutation representations on the ordered pairs of points from a set of cardinality $m$. Since $m=176$ is even, these permutation representations embed in $\Alt({m\choose 2})=\Alt(\Omega)$, see Lemma~\ref{l:2}. From this, we deduce that $\mathcal{O}_{\Sym(\Omega)}(H)$ is not Boolean because $\Alt(\Omega)$ is the only maximal element of $\mathcal{O}_{\Sym(\Omega)}(H)$. When $G=\Alt(\Omega)$, $\mathcal{O}_G(M)$ has three maximal elements and one of these maximal elements is $\Aut(H)\cong HS.2$. If $\mathcal{O}_G(H)$ is Boolean, then it has rank $3$ and hence $\mathcal{O}_{HS.2}(HS)$ is Boolean of rank $2$: however this is a contradiction because $|\Aut(HS):HS|=|HS.2:HS|=2$, see Lemma~\ref{l:-01}.

In Case~\eqref{thrmAenu2}~(b), the group $\Aut(\Omega_7(3))\cong \Omega_7(3).2$ has no faithful permutation representations of degree $3159$. Since $|\mathcal{M}_{\Sym(\Omega)}(H)|=3$, we deduce $\mathcal{M}_{\Sym(\Omega)}(H)$ contains two subgroups isomorphic to $\Omega_7(3)$ which are contained in $\Alt(\Omega)$ and $\Aut(U)\cong \mathrm{G}_2(3).2$ which is not contained in $\Alt(\Omega)$ (the fact that $\mathrm{G}_2(3).2\nleq \Alt(\Omega)$ can be easily verified with the computer algebra system \texttt{magma}~\cite{magma}).  When $G=\Sym(\Omega)$, we obtain that $\mathcal{O}_G(H)$ is not Boolean.
When $G=\Alt(\Omega)$, we were not able to determine whether $\mathcal{O}_G(H)$ is Boolean, but if it is Boolean, then it has rank $2$ having maximal elements two subgroups isomorphic to $\Omega_7(3)$.

In Case~\eqref{thrmAenu2}~(c) and $n=12$, we see that $M_{12}.2$ does not admit a permutation representation of degree $12$. Therefore, as above, since $|\mathcal{M}_{\Sym(\Omega)}(H)|=3$, we deduce that $\mathcal{M}_{\Sym(\Omega)}(H)$ contains two subgroups isomorphic to $M_{12}$ which are contained in $\Alt(\Omega)$ and $\Aut(U)\cong \mathrm{PGL}_2(11)$ which is not contained in $\Alt(\Omega)$. Therefore, $\mathcal{O}_{\Sym(\Omega)}(H)$ is not Boolean.  When $G=\Alt(\Omega)$, we have verified with the help of a computer that $\mathcal{O}_G(H)$ is indeed Boolean of rank $2$. In Case~\eqref{thrmAenu2}~(c) and $n=24$, we see that $\Aut(M_{24})=M_{24}$. Therefore, since $|\mathcal{M}_{\Sym(\Omega)}(H)|=3$, we deduce $\mathcal{M}_{\Sym(\Omega)}(H)$ contains two subgroups isomorphic to $M_{24}$ which are contained in $\Alt(\Omega)$ and $\Aut(U)\cong \mathrm{PGL}_2(23)$ which is not contained in $\Alt(\Omega)$. Therefore, $\mathcal{O}_{\Sym(\Omega)}(H)$ is not Boolean. When $G=\Alt(\Omega)$, we have verified with the help of a computer that $\mathcal{O}_G(H)$ is indeed Boolean of rank $2$.

In Case~\eqref{thrmAenu2}~(d), we see that $\Aut(\mathrm{Sp}_8(2))=\mathrm{Sp}_{8}(2)$. Therefore, since $|\mathcal{M}_{\Sym(\Omega)}(H)|=3$, we deduce that $\mathcal{M}_{\Sym(\Omega)}(H)$ contains two subgroups isomorphic to $\mathrm{Sp}_{8}(2)$ which are contained in $\Alt(\Omega)$ and $\Aut(U)\cong \mathrm{PGL}_2(17)$ which is not contained in $\Alt(\Omega)$. Therefore, $\mathcal{O}_{\Sym(\Omega)}(H)$ is not Boolean. When $G=\Alt(\Omega)$, we were not able to determine whether $\mathcal{O}_G(H)$ is Boolean, but if it is Boolean, then it has rank $2$.

\item In Case~\eqref{thrmAenu3}, we use a computer to deal with this case. None of the four elements in $\mathcal{M}_{\Sym(\Omega)}(U)$ is contained in $\Alt(\Omega)$. Therefore, if $\mathcal{O}_{\Sym(\Omega)}(H)$ is Boolean, then it has rank $4$. Moreover, the intersection of these four subgroups is $H$ and we see that $|H:U|=2$. As $|\Aut(\mathrm{PSL}_3(4)):\mathrm{PSL}_3(4)|=12$, we deduce $|\nor {\Sym(\Omega)}U:H|=6=2\cdot 3$. 
Therefore $\mathcal{O}_{\nor {\Sym(\Omega)}H}(H)$ cannot be a rank $3$ Boolean 
lattice (see Lemma~\ref{l:-01}), contradicting the fact that we assumed $\mathcal{O}_{\Sym(\Omega)}(H)$ to be Boolean. Assume then $G=\Alt(\Omega)$. Define 
$M_0:=\nor{\Alt(\Omega)}U$ and let $M_1,M_2,M_3$ be the intersections with $\Alt(\Omega)$ of the three maximal subgroups of $\Sym(\Omega)$ isomorphic to $\mathrm{Aut}(\mathrm{PSU}_4(3))$. Assume that $\mathcal{O}_{\Alt(\Omega)}(H)$ is Boolean. If $H<M_0$, then $\mathcal{O}_{\Alt(\Omega)}(H)$ is Boolean of rank $4$ and hence $\mathcal{O}_{M_0}(H)$ is Boolean of rank $3$. However this is impossible because $|M_0:U|=6=2\cdot 3$. Therefore $H=M_0=\nor {\Alt(\Omega)}U$. However this is another contradiction because $M_0$ is maximal in $\Alt(\Omega)$.

\item In Case~\eqref{thrmAenu4}, $k$ is odd and hence $H$ is a subgroup of $\Alt(\Omega)$. The action under consideration arises using the standard $2$-transitive action of $\mathrm{Sz}(q)$ of degree $q^2+1$. Here, the action of degree $q^2(q^2+1)/2$ is the action on the pairs of points from the set $\{1,\ldots,q^2+1\}$. Here $K_1\nleq \Alt(\Omega)$ because $q^2+1$ is odd, see Lemma~\ref{l:2}. Moreover, $\Aut(\mathrm{Sp}_{4k}(2))=\mathrm{Sp}_{2k}(2)$ and $K_2=V_2$, hence $V_2\le \Alt(\Omega)$. From this we deduce that the maximal elements in $\mathcal{O}_{\Sym(\Omega)}(H)$ are $K_1\cong\Sym(q^2+1)$ and $\Alt(\Omega)$. However this lattice is not Boolean because $H\ne K_1\cap \Alt(\Omega)=V_1\cong\Alt(q^2+1)$. When $G=\Alt(\Omega)$, the maximal elements in $\mathcal{O}_G(H)$ are $V_1\cong \Alt(q^2+1)$ and $V_2\cong \mathrm{Sp}_{4k}(2)$. Therefore, if $\mathcal{O}_G(H)$ is Boolean, then its rank is $2$.

\item  In Case~\eqref{thrmAenu5}, $\mathcal{O}_{\Sym(\Omega)}(H)$ is not Boolean because $\mathcal{O}_K(H)$ is not Boolean. When $G=\Alt(\Omega)$, $\mathcal{O}_{\Alt(\Omega)}(H)$ contains two maximal elements $V$ and $V^t$ both isomorphic to $\Alt(11)$. Therefore, if $\mathcal{O}_{\Alt(11)}(H)$ were Boolean, then $\mathcal{O}_{\Alt(\Omega)}(H)$ would have rank $2$. However this is not the case because $\mathcal{O}_V(H)=\{H<M<V\}$ and $\mathcal{O}_{V^t}=\{H<M^t<V^t\}$ with $M\cong M^t\cong M_{11}$. Therefore $\mathcal{O}_{\Alt(\Omega)}(H)$ is not Boolean.
\end{enumerate}}
\end{remarks}

\begin{corollary}\label{cor:1}
Let $H$ be an almost simple primitive subgroup of $G$ which is product indecomposable and not octal. If $\mathcal{O}_{G}(H)$ is Boolean, then it has rank at most $2$.
\end{corollary}
\begin{proof}
It follows from Theorem~\ref{thrm:A} and Remark~\ref{rem:1}. 
\end{proof}

\section{Boolean intervals $\mathcal{O}_G(H)$ with $H$ primitive}\label{sec:primitive}

\begin{lemma}\label{l:orcoboia}
Let $M$ be a maximal subgroup of $G$ of O'Nan-Scott type $\mathrm{SD}$ and let  $H$ be a maximal subgroup of $M$ acting primitively on $\Omega$. Then $M$ and $H$ have the same socle.
\end{lemma}
\begin{proof}
This follows from~\cite[Theorem]{18} (using the notation in~\cite{18}, applied with $G_1:=M$, see also~\cite[Proposition~$8.1$]{18}).
\end{proof}

\begin{lemma}\label{c:orcoboia}
Let $H$ be a primitive subgroup of $G$ with $\mathcal{O}_G(H)$ Boolean of rank $\ell$. Suppose that there exists a maximal element $M\in\mathcal{O}_G(H)$ of O'Nan-Scott type $\mathrm{SD}$. Then $\ell\le 2$.
\end{lemma}
\begin{proof}
Let $V$ be the socle of $M$. From the structure of primitive groups of SD type, we deduce $V\cong T^\kappa$ and $|\Omega|=|T|^{\kappa-1}$, for some non-abelian simple group $T$ and for some integer $\kappa\ge 2$. 

If $\ell=1$, then we have nothing to prove, therefore we suppose $\ell\ge 2$ and we let $M'\in\mathcal{O}_G(H)$ be a maximal element with $M'\ne M$. Set $H':=M\cap M'$. Since $\mathcal{O}_G(H)$ is Boolean, $H'$ is maximal in $M$ and since $H\le H'$, $H'$ acts primitively on $\Omega$. From Lemma~\ref{l:orcoboia} applied with $H$ there replaced by $H'$ here, we obtain that $H'$  has socle $V$. From the O'Nan-Scott theorem and in particular from the structure of the socles of primitive groups, we deduce that $H'$ has type HS or SD, where the type HS can arise only when $\kappa=2$.  Now, from~\cite[Proposition~$8.1$]{18}, we obtain that  either $M'$ is a primitive group of SD type having socle $V$, or $M'=\Alt(\Omega)$.  In the first case, $M'=\nor G V= M$, which is a contradiction. Therefore $M'=\Alt(\Omega)$. Thus $G=\Sym(\Omega)$ and $\Alt(\Omega)$ and $M$ are the only maximal members in $\mathcal{O}_G(H)$. This gives $\ell= 2$.
\end{proof}

\begin{lemma}\label{l:orcoboiaboia}
Let $M$ be a maximal subgroup of $G$ of O'Nan-Scott type $\mathrm{HA}$ with socle $V$ and let $H$ be a maximal subgroup of $M$ acting primitively on $\Omega$. Then either
\begin{enumerate}
\item $V\le H$, or
\item $|\Omega|=8$, $G=\Alt(\Omega)$, $H\cong \mathrm{PSL}_2(7)$ and $M\cong \mathrm{AGL}_3(2)$.
\end{enumerate}
\end{lemma}
\begin{proof}
Here, $n=|\Omega|=p^d$, for some prime number $p$ and some positive integer $d$. The result is clear when $n\le 4$ and hence we suppose $n\ge 5$.
In what follows, we assume $V\nleq H$ and we show that $n=8$, $G=\Alt(\Omega)$, $H\cong \mathrm{PSL}_2(7)$ and $M\cong \mathrm{AGL}_3(2)$. 

The maximality of $H$ in $M$ yields $VH=M$. Since $V\cap H\unlhd \langle V,H\rangle=M$, we deduce $V\cap H=1$, that is, $H$ is a complement of $V$ in $M$ and hence $H\cong M/V$. Since $\nor {\Sym(n)}V\cong \mathrm{AGL}_d(p)$, we deduce $M/V$ and $H$ are isomorphic to $\mathrm{GL}_d(p)$ or to an index $2$ subgroup of $\mathrm{GL}_d(p)$.

Since $H$ acts primitively on $\Omega$, we deduce  ${\bf Z}(H)=1$ or ${\bf Z}(H) =H$.   Clearly, the second case cannot arise here because $M/V$ is non-abelian being $n\ge 5$. Suppose then ${\bf Z}(H)=1$. 

If $G=\Sym(\Omega)$, then $M/V\cong \mathrm{GL}_d(p)$ has trivial centre only when $p=2$. It is easy to verify (using the fact that $\mathrm{GL}_d(2)$ is generated by transvections) that $\mathrm{AGL}_d(2)$ is contained in $\Alt(\Omega)$, when $d\ge 3$. Thus $M<\Alt(\Omega)<G$, contradicting the hypothesis that $M$ is maximal in $G$. This shows that $G=\Alt(\Omega)$. In particular, when  $p=2$, we have $M/V\cong \mathrm{GL}_d(2)$ and when $p>2$, $M/V$ is isomorphic to a subgroup of $\mathrm{GL}_d(p)$ having index $2$. 

Since $\mathrm{GL}_d(p)$ has centre of order $p-1$ and since ${\bf Z}(H)=1$, we deduce that either $p=2$ or $(p-1)/2=1$, that is, $p\in \{2,3\}$. In both cases, a simple computation reveals that $M=\textrm{ASL}_d(p)$ and hence $H\cong M/V\cong \mathrm{SL}_d(p)$. Observe that, when $p=3$, $d$ is odd  because $1=|{\bf Z}(H)|=|{\bf Z}(\mathrm{SL}_d(3))|=\gcd(d,2)$. In particular, in both cases, $H\cong M/V\cong \mathrm{SL}_d(p)\cong \mathrm{PSL}_d(p)$ is a non-abelian simple group. Given $\omega\in \Omega$, $|H:H_\omega|=p^d$ is a power of the prime $p$ and hence, from~\cite[(3.1)]{hoffman}, 
we deduce $(d,p)=(3,2)$. Thus $n=p^d=8$, $H\cong \mathrm{SL}_3(2)\cong \mathrm{PSL}_2(7)$.
\end{proof}

\begin{lemma}\label{c:orcoboiaboia}
Let $H$ be a primitive subgroup of $G$ with $\mathcal{O}_G(H)$ Boolean of rank $\ell$. Suppose that there exists a maximal element $M\in\mathcal{O}_G(H)$ of O'Nan-Scott type $\mathrm{HA}$. Then, every maximal element $M'$ in $\mathcal{O}_G(H)$ with $M'\ne M$ is either $\Alt(\Omega)$ or the stabilizer in $G$ of a regular product structure on $\Omega$.
\end{lemma}
\begin{proof}
If $\ell=1$, then we have nothing to prove, therefore we suppose that $\ell\ge 2$ and we let $M'\in\mathcal{O}_G(H)$ be a maximal element of $\mathcal{O}_G(H)$ with $M'\ne M$. Set $H':=M\cap M'$. Since $\mathcal{O}_G(H)$ is Boolean, $H'$ is maximal in $M$ and since $H\le H'$, $H'$ acts primitively on $\Omega$. From Lemma~\ref{l:orcoboiaboia} applied with $H$ replaced by $H'$, we obtain that either $H'$ contains the socle $V$ of $M$, or $n=8$, $G=\Alt(\Omega)$, $H'\cong \mathrm{PSL}_2(7)$ and $M\cong\mathrm{AGL}_3(2)$. In the second case, a computer computation reveals that the largest Boolean lattice  $\mathcal{O}_{\Alt(8)}(H)$ with $H$ primitive has rank $2$.  Therefore, for the rest of the proof, we suppose $V\le M'$. In particular, $M'$ is a primitive permutation group containg an abelian regular subgroup. Thus $M'$ is one of the groups classified in~\cite[Theorem~$1.1$]{li}: we apply this classification here and the notation therein. 

Assume  $M'$ is as in~\cite[Theorem~$1.1$~(1)]{li}, that is, $M'$ is a maximal primitive subgroup of $G$ of O'Nan-Scott type HA. Let $V'$ be the socle of $M'$. From Lemma~\ref{l:orcoboiaboia}, we deduce $V'\le M$ and hence $VV'\le H'$. Since $V\unlhd M$ and $V'\unlhd M'$, we deduce that $VV'\unlhd H'$. As $H'$ acts primitively on $\Omega$, we deduce that $VV'$ is the socle of $H'$ and hence $|VV'|=|V|$. Therefore $V=V'$. Thus $M'=\nor G V=M$, which is a contradiction. Therefore $M'$ is one of the groups listed in~\cite[Theorem~$1.1$~(2)]{li}.

Suppose first that $l=1$ (the positive integer $l$ is defined in~\cite[Theorem~$1.1$]{li}). An inspection in the list in~\cite[Theorem~$1.1$~(2)]{li} (using the maximality of $M'$ in $G$) yields 
\begin{enumerate}
\item\label{again:1} $M'\cong M_{11}$, $n=11$ and $G=\Alt(\Omega)$, or
\item\label{again:2} $M'\cong M_{23}$, $n=23$  and $G=\Alt(\Omega)$, or
\item\label{again:3} $M'\cong \nor G {\mathrm{PSL}_{d'}(q')}$ for some integer $d'\ge 2$ and some prime power $q'$ with $n=p=(q'^{d'}-1)/(q'-1)$, or
\item\label{again:4} $M'=\Alt(\Omega)$ and $G=\Sym(\Omega)$.
\end{enumerate}
A computer computation shows that in \eqref{again:1} and~\eqref{again:2}, $M=\nor G{V}\le M'$, which is a contradiction. Assume that $M'$ is as in~\eqref{again:3}. Write $q'=r'^{\kappa'}$, for some prime number $r'$ and for some positive integer $\kappa'$. Then $V$ is a Singer cycle in $\mathrm{PGL}_{d'}(q')$. As $H'=M\cap M'=\nor G V\cap M'=\nor {M'}V$, we obtain
\[
|H':V|=
\begin{cases}
d'\kappa',& \textrm{when }\nor{\Sym(p)}{\mathrm{PGL}_{d'}(q')}\le G,\\
d'\kappa'/2,& \textrm{when } \nor{\Sym(p)}{\mathrm{PGL}_{d'}(q')}\nleq G.
\end{cases}
\] We claim that $d'$ is prime. If $d'$ is not prime, then $d'=d_1d_2$, for some positive integers $d_1,d_2>1$. Thus $H'<\nor G{\mathrm{PSL}_{d_1}(q'^{d_2})}<M'$, contradicting the fact that $H'$ is maximal in $M'$.  Therefore, $d'$ is a prime number. Moreover, since $H'$ is maximal in $M$ and since $M/V$ is cyclic (of order $p-1$ or $(p-1)/2$), we deduce that $s':=|M:H'|$ is a prime number.

 Let $M''$ be a maximal element in $\mathcal{O}_G(H)$ with $M\ne M''\ne M'$ and let $H'':=M\cap M''$. Arguing as in the previous paragraph (with $M'$ replaced by $M''$), $M''$ cannot be as in~\eqref{again:1} or as in~\eqref{again:2}. Suppose that $M''$ is as in~\eqref{again:3}. Thus $M''\cong \nor G {\mathrm{PSL}_{d''}(q'')}$ for some integer $d''\ge 2$ and some prime power $q''$ with $n=p=(q''^{d''}-1)/(q''-1)$. Write $q''=r''^{\kappa''}$, for some prime number $r''$ and for some positive integer $\kappa''$. Arguing as in the previous paragraph, we obtain that $d''$ and $s'':=|M:H''|$ are prime numbers. Now, $M'\cap M''$ acts primitively on $\Omega$ with $n=|\Omega|=p$ prime and hence, from a result of Burnside, $M'\cap M''$ is either soluble (and $V\unlhd M'\cap M''$) or $M'\cap M''$ is $2$-transitive. In the first case, $M\cap M'=\nor {M'}V\ge M'\cap M''$; however, this contradicts the fact that $\mathcal{O}_G(H)$ is Boolean. Therefore $M'\cap M''$ is $2$-transitive and non-soluble. From~\cite[Theorem~$3$]{gareth}, we deduce that one of the following holds:
\begin{enumerate}
\item $M'\cap M''=\mathrm{PSL}_2(11)$ and $n=p=11$, or
\item $M'\cap M''=M_{11}$ and $n=p=11$, or
\item $M'\cap M''=M_{23}$ and $n=p=23$, or 
\item $M'\cap M''\unlhd M'$ and $M'\cap M''\unlhd M''$.
\end{enumerate}
The last case cannot arise because $M'\cap M''\unlhd \langle M',M''\rangle=G$ implies $M'\cap M''=1$, which is a contradiction.  Also none of the first three cases can arise here because $p$ is not of the form $(q'^{d'}-1)/(q'-1)$. This final contradiction shows that, if $M''$ is a maximal element of $\mathcal{O}_G(H)$ with $M''\notin\{M,M'\}$, then $M''=\Alt(\Omega)$. Thus $\ell=3$, $|\Omega|=p$, $G=\Sym(\Omega)$ and the maximal elements in $\mathcal{O}_G(H)$ are $\Alt(\Omega)$, $\mathrm{AGL}_1(p)$ and $\mathrm{P}\Gamma\mathrm{L}_{d'}(q')$. 

Since $M\cong \mathrm{AGL}_1(p)$, $|H':V|=d'\kappa'$ and $|M:H'|=s'$ is prime, we obtain 
\begin{equation}\label{eq:number}
q'\frac{q'^{d'-1}-1}{q'-1}=p-1=|M:V|=|M:H'||H':V|=s'd'\kappa'.
\end{equation}
Suppose first that $d'=2$ and hence $p=q'+1=r'^{\kappa'}+1$. We get the equation $r'^{\kappa'}=2s'\kappa'$ and hence $r'=2$. Therefore $2^{\kappa'-1}=s'\kappa'$. Therefore, $s'=2$ and hence $2^{\kappa'-2}=\kappa'$. Thus $\kappa'=4$ and hence $n=p=17$. A computer computation shows that this case does not arise because $\Alt(17)\cap \mathrm{AGL}_1(17)=\mathrm{AGL}_1(17)\cap\mathrm{P}\Gamma\mathrm{L}_2(16)$. Suppose now $d'>2$. 

Assume $\kappa'=1$. Then~\eqref{eq:number} yields $s'=2$ because $p-1$ is even. A computation shows that the equation $$q'\frac{q'^{d'-1}-1}{q'-1}=2d'$$
has solution only when $d'=3$ and $q'=2$. Thus $n=p=7$. A computer computation shows that this case does not arise because $\Alt(7)\cap \mathrm{AGL}_1(7)\le \Alt(7)\cap\mathrm{PGL}_2(7)$. Therefore $\kappa'>1$.

Now, we first show $d'\ne r'$. To this end, we argue by contradiction and we suppose $d'=r'$. Then,~\eqref{eq:number} yields
\begin{equation}\label{eq:number1}
\frac{q'}{r'}\frac{q'^{d'-1}-1}{q'-1}=s'\kappa'.
\end{equation}
Since $q'/r'=r'^{\kappa'-1}$ and $(q'^{d'-1}-1)/(q'-1)$ are relatively prime and since $s'$ is prime, we have either $s'=r'$ or $s'$ divides $(q'^{d'-1}-1)/(q'-1)$. In the first case, 
$$\frac{q'}{r'^2}\frac{q'^{d'-1}-1}{q'-1}=\kappa',$$
hence, for $d'>3$, $\kappa'\ge (q'^{d'-1}-1)/(q'-1)\ge q'^{2}=r'^{2\kappa'}$, which is impossible. It is not difficult to observe that~\eqref{eq:number1} is not satisfied also for $d'=3$. In the second case, 
$\kappa'\ge q'/r'=r'^{\kappa'-1}$, which is possible only when $\kappa'=2.$
When $\kappa'=2$,~\eqref{eq:number1} becomes
$$r'\frac{(r'^{d-1}-1)(r'^{d-1}+1)}{r'^2-1}=2s',$$
which has no solution with $s'$ prime. Therefore $d'\ne r'$.

Since $d'$ is a prime number and since $d'\ne r'$, from Fermat's little theorem we have $q'^{d'-1}\equiv 1\pmod {d'}$, that is, $d'$ divides $q'^{d'-1}-1$. If $q'\equiv 1\pmod {d'}$, then $$p=\frac{q'^{d'}-1}{q'-1}=q'^{d'-1}+q'^{d'-2}+\cdots +\cdots +q'+1\equiv 0\pmod{d'}$$
and hence $p=d'$, however this is clearly a contradiction because $p>d'$. Thus $d'$ does not divide $q'-1$. This proves that $d'$ divides $(q'^{d'-1}-1)/(q'-1)$ and hence $(q'^{d'-1}-1)/d'(q'-1)$ is an integer. From~\eqref{eq:number}, we get
$$q'\frac{q'^{d'-1}-1}{d'(q'-1)}=\kappa's'.$$  
 Since $s'$ is prime and $q'>\kappa'$, this equality might admit a solution only when $(q'^{d'-1}-1)/(d'(q'-1))=1$, that is, $q'^{d'-1}-1=d'(q'-1)$. This happens only when $q'=2$ and $d'=3$, but this contradicts $\kappa'>1$.

\medskip

For the rest of the argument we may suppose $l\ge 2$. In particular, from~\cite[Theorem~$1.1$]{li}, we obtain that either  $M'=\Alt(\Omega)$, or $M'$ is the stabilizer in $G$ of a regular product structure on $\Omega$. Since this argument does not depend upon $M'$, the result follows.
\end{proof}

\begin{lemma}\label{l:orcoboia3}
Let $M$ be a maximal subgroup of $G$ of O'Nan-Scott type $\mathrm{AS}$ with $M\ne \Alt(\Omega)$ and let $H$ be a maximal subgroup of $M$ acting primitively on $\Omega$. Then 
\begin{enumerate}
\item $M$ and $H$ have the same socle, or 
\item $H$ has O'Nan-Scott type $\mathrm{AS}$ and the pair $(H,M)$ appears in Tables~$3$--$6$ of~\cite{16}, or
\item $H$ has O'Nan-Scott type $\mathrm{HA}$ and the pair $(H,M)$ appears in Table~$2$ of~\cite{18}.
\end{enumerate}
\end{lemma}
\begin{proof}
Suppose that $H$ and $M$ do not have the same socle. It follows from~\cite[Proposition~$6.2$]{18} that  either $H$ has O'Nan-Scott type AS and the pair $(H,M)$ appears in Tables~$3$--$6$ of~\cite{16}, or  $H$ has O'Nan-Scott type HA and the pair $(H,M)$ appears in Table~$2$ of~\cite{18}.
\end{proof}

\begin{lemma}\label{c:orcoboiaboia3}
Let $H$ be a primitive subgroup of $G$ with $\mathcal{O}_G(H)$ Boolean of rank $\ell$. Suppose that there exists a maximal element $M\in\mathcal{O}_G(H)$ of O'Nan-Scott type $\mathrm{AS}$ with $M\ne \Alt(\Omega)$. Then $\ell\le 2$.
\end{lemma}
\begin{proof}
If $\ell\le 2$, then we have nothing to prove, therefore we suppose $\ell\ge 3$. Since $M$ is a maximal element in $\mathcal{O}_G(H)$ of O'Nan-Scott type AS and $M\ne \Alt(\Omega)$, from Lemma~\ref{c:orcoboiaboia}, we deduce that no maximal element in $\mathcal{O}_G(H)$ is of O'Nan-Scott type HA. Similarly, from Lemma~\ref{c:orcoboia}, no maximal element in $\mathcal{O}_G(H)$ is of O'Nan-Scott type SD. As $H$ acts primitively on $\Omega$, all elements in $\mathcal{O}_G(H)$ are primitive and hence, the maximal elements in $\mathcal{O}_G(H)$ have O'Nan-Scott type AS or PA. Since $\ell\ge 3$, we let $M'\in\mathcal{O}_G(H)$ be a maximal element with $\Alt(\Omega)\ne M'\ne M$. Moreover, we let $M''$ be any maximal element in $\mathcal{O}_G(H)$ with $M\ne M''\ne M'$. Set $H':=M\cap M'$ and $H'':=M\cap M'\cap M''$. See Figure~\ref{fig3}.
\begin{figure}[!ht]
\begin{tikzpicture}[node distance=1cm]
\node(A0){$G$};
\node(A1)[below of=A0]{$M$};
\node(A2)[left of=A1]{$M'$};
\node(A3)[right of=A1]{$M''$};
\node(A4)[below of=A1]{$H'$};
\node(A5)[below of=A2]{$M'\cap M''$};
\node(A6)[below of=A3]{$M\cap M''$};
\node(A7)[below of=A4]{$M\cap M'\cap M''$};
\draw(A0)--(A1);
\draw(A0)--(A2);
\draw(A0)--(A3);
\draw(A1)--(A5);
\draw(A1)--(A6);
\draw(A2)--(A5);
\draw(A2)--(A4);
\draw(A3)--(A6);
\draw(A3)--(A4);
\draw(A7)--(A6);
\draw(A7)--(A5);
\draw(A7)--(A4);
\draw(A1)--(A4);
\end{tikzpicture}
\caption{The Boolean lattice in the proof of Lemma~\ref{c:orcoboiaboia3}} \label{fig3}
\end{figure}
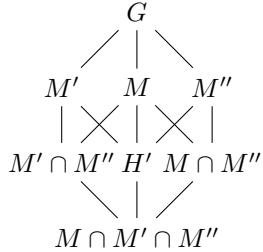

 Since $\mathcal{O}_G(H)$ is Boolean, $H'$ is maximal in $M$ and hence we are in the position to apply Lemma~\ref{l:orcoboia3} with $H$ there replaced by $H'$ here. We discuss the three possibilities in turn.

Suppose first that $H'$ has O'Nan-Scott type HA and let $V'$ be the socle of $H'$. Since in $\mathcal{O}_G(H)$ there are no maximal members of HA type, $\nor G {V'}$ is not a maximal subgroup of $G$. It follows from~\cite[Theorem]{16} that $n\in \{7,11,17,23\}$ and $G=\Alt(\Omega)$. A computer computation shows that none of these cases gives rise to a Boolean lattice of rank $3$ or larger.

Suppose now that $H'$ and $M$ have the same socle, or that the pair $(H',M)$ appears in Tables~$3$--$6$ of~\cite{16}. In these cases, $H'$ has O'Nan-Scott type AS. Since $\mathcal{O}_G(H)$ is Boolean, $H''$ is maximal in $H'$ and hence,  from Lemma~\ref{l:orcoboia3}, either 
\begin{itemize}
\item $H''$ and $H'$ have the same socle, 
\item or $H''$ has O'Nan-Scott type AS and the pair $(H'',H')$ appears in Tables~$3$--$6$ of~\cite{16}, 
\item or  $H''$ has O'Nan-Scott type HA and the pair $(H'',H')$ appears in Table~$2$ of~\cite{18}. 
\end{itemize}Suppose first that $H''$ has O'Nan-Scott type HA and let $V''$ be the socle of $H''$. Since in $\mathcal{O}_G(H)$ there are no maximal members of HA type, $\nor G {V''}$ is not a maximal subgroup of $G$, as above. It follows from~\cite[Theorem]{16} that $n\in \{7,11,17,23\}$ and $G=\Alt(\Omega)$. The same computer computation as above shows that none of these cases gives rise to a Boolean lattice of rank $3$ or larger. Therefore, $H''$ has O'Nan-Scott type AS.

As $\mathcal{O}_G(H'')$ has rank $3$ and $H''$ has type AS, Corollary~\ref{cor:1} implies that $H''$ is either product decomposable or octal. If $H''$ is octal, then $n=8$ and $H''\cong \mathrm{PSL}_2(7)$, however the largest Boolean lattice containing $H''$ has rank $2$. Thus $H''$ is product decomposable. 

From~\cite[Table~II]{16}, one of the following holds:
\begin{enumerate}
\item $n=36$, $H''=\Alt(6).2$,
\item $n=144$, $H''=M_{12}.2$, 
\item $n=q^4(q^2-1)^2/4$ and ${\bf F}^*(H'')=\mathrm{Sp}_4(q)$, where $q>2$ is even.
\end{enumerate} 
When $n=144$ and $H''=M_{12}.2$, we see that $H'$ cannot have the same socle as $H''$ because $H''\cong \Aut(M_{12})$ and hence $(H'',H')$ is one of the pairs in Tables~$3$--$6$ of~\cite{16}. However, there is no such pair satisfying $n=144$ and ${\bf F}^*(H'')\cong M_{12}$. When $n=36$ and $H''=\Alt(6).2$, we see with a computer computation that 
$\nor {\Sym(36)}{H''}=H''$ and hence $H'$ cannot have the same socle as $H''$. Therefore $(H'',H')$ is one of the pairs in Tables~$3$--$6$ of~\cite{16}. However, there is no such pair satisfying $n=36$ and ${\bf F}^*(H'')\cong \Alt(6)$. Finally, suppose $n=q^4(q^2-1)^2/4$ and ${\bf F}^*(H'')=\mathrm{Sp}_4(q)$, where $q>2$ is even. Since there is no pair $(H'',H')$ in Tables~$3$--$6$ of~\cite{16} satisfying these conditions for $n$ and ${\bf F}^*(H'')$ as above, we deduce that $H''$ and $H'$ have the same socle. Therefore ${\bf F}^*(H')=\mathrm{Sp}_4(q)$, with $q>2$ even. 

Summing up, we have two inclusions $H'\le M$ and $H'\le M'$, with $H'$ maximal in both $M$ and $M'$, with ${\bf F}^*(H')=\mathrm{Sp}_4(q)$ and with $n=q^4(q^2-1)^2/4$. Using again Tables~$3$--$6$ of~\cite{16}, we deduce that both $M$ and $M'$ must have the same socle of $H'$. However, this is a contradiction because $G=\langle M,M'\rangle \le \nor G{{\bf F}^*(H')}$.
\end{proof}

\begin{corollary}\label{cor:2}
Let $H$ be a primitive subgroup of $G$ with $\mathcal{O}_G(H)$ Boolean of rank $\ell\ge 3$ and let $G_1,\ldots,G_\ell$ be the maximal members in $\mathcal{O}_G(H)$. Then one of the following holds:
\begin{enumerate}
\item\label{accent1} $n=|\Omega|$ is odd. For every $i\in \{1,\ldots,\ell\}$, there exists a non-trivial regular product structure $\mathcal{F}_i$ with $G_i=\nor G{\mathcal{F}_i}$; moreover, relabeling the indexed set $\{1,\ldots,\ell\}$ if necessary, $\mathcal{F}_1< \cdots < \mathcal{F}_\ell$.
\item\label{accent2}$n=|\Omega|$ is odd and $G=\Sym(\Omega)$. Relabeling the indexed set $\{1,\ldots,\ell\}$ if necessary, $G_\ell=\Alt(\Omega)$, for every $i\in \{1,\ldots,\ell-1\}$, there exists a non-trivial regular product structure $\mathcal{F}_i$ with $G_i=\nor G{\mathcal{F}_i}$; moreover, relabeling the indexed set $\{1,\ldots,\ell-1\}$ if necessary, $\mathcal{F}_1< \cdots < \mathcal{F}_{\ell-1}$.
\item\label{accent3}$n=|\Omega|$ is an odd prime power. Relabeling the indexed set $\{1,\ldots,\ell\}$ if necessary, $G_\ell$ is maximal subgroup of O'Nan-Scott type HA, for every $i\in \{1,\ldots,\ell-1\}$, there exists a non-trivial regular product structure $\mathcal{F}_i$ with $G_i=\nor G{\mathcal{F}_i}$; moreover, relabeling the indexed set $\{1,\ldots,\ell-1\}$ if necessary, $\mathcal{F}_1< \cdots < \mathcal{F}_{\ell-1}$.
\item\label{accent4} $n=|\Omega|$ is an odd prime power and $G=\Sym(\Omega)$. Relabeling the indexed set $\{1,\ldots,\ell\}$ if necessary, $G_\ell=\Alt(\Omega)$ and $G_{\ell-1}$ is a maximal subgroup of O'Nan-Scott type HA, for every $i\in \{1,\ldots,\ell-2\}$, there exists a non-trivial regular product structure $\mathcal{F}_i$ with $G_i=\nor G{\mathcal{F}_i}$; moreover, relabeling the indexed set $\{1,\ldots,\ell-2\}$ if necessary, $\mathcal{F}_1< \cdots < \mathcal{F}_{\ell-2}$.
\end{enumerate}
\end{corollary}
\begin{proof}
As $\ell\ge 3$, from Lemmas~\ref{c:orcoboia},~\ref{c:orcoboiaboia} and~\ref{c:orcoboiaboia3}, all the elements in $\{G_1,\ldots,G_\ell\}$ are stabilizers of regular product structures, except possibly that one of these elements might be $\Alt(\Omega)$ or  a maximal subgroup of type HA. Relabelling the indexed set $\{1,\ldots,\ell\}$, suppose that $\{G_1,\ldots,G_\kappa\}$ are stabilizers of regular product structures, that is, $G_i:=\nor G{\mathcal{F}_i}$. Thus $\kappa\ge \ell-2$.

Observe that, for every $i,j\in \{1,\ldots,\kappa\}$ with $i\ne j$, $G_i\cap G_j$ is a maximal subgroup of both $G_i$ and $G_j$. It follows from~\cite[Section~$5$]{2} that either $\mathcal{F}_i<\mathcal{F}_j$ or $\mathcal{F}_j< \mathcal{F}_i$. Therefore $\{\mathcal{F}_1,\ldots,\mathcal{F}_\kappa\}$ forms a chain. Relabeling the indexed set $\{1,\ldots,\kappa\}$ if necessary, we may suppose $\mathcal{F}_1< \mathcal{F}_2< \cdots <\mathcal{F}_\kappa$.

Assume that $\mathcal{F}_i$ is a regular $(m_i,k_i)$-product structure. Since $\mathcal{F}_i\le \mathcal{F}_{i+1}$, there exists an integer $s_i>1$ with $m_i=m_{i+1}^{s_i}$ and $k_{i+1}=k_is_i$. From~\cite[(5.12)]{2}, $
\mathcal{O}_G(
\nor G{\mathcal{F}_i}
\cap 
\nor G{\mathcal{F}_{i+1}}
)$ is Boolean of rank $2$ only when $$(\dag)\quad m_{i+1} \textrm{ is odd},\, \textrm{or }s_i=2\textrm{ and }m_{i+1}\equiv 2\pmod 4.$$

Suppose that $\kappa\ge 3$. Applying the previous paragraph with $i:=\kappa-1$, we deduce that, if $m_\kappa$ is even, then $s_{\kappa-1}=2$ and $m_{\kappa}\equiv 2\pmod 4$. In turn, since $m_{\kappa-1}=m_\kappa^{s_{\kappa-1}}$ is even, we have $s_{\kappa-2}=2$ and $m_{\kappa-1}\equiv 2\pmod 4$. However, $m_{\kappa-1}=m_\kappa^{s_{\kappa-1}}\equiv 0\pmod 4$, contradicting the fact that $m_{\kappa-1}\equiv 2\pmod 4$.
Therefore, when $\kappa\ge 3$, $m_i$ is odd, for every $i\in\{1,\ldots,\kappa\}$, that is, $n=|\Omega|$ is odd. In particular, when $\kappa=\ell$, we obtain part~\eqref{accent1}.

Suppose that $G=\Sym(\Omega)$, $\kappa=\ell-1$ and $G_\ell=\Alt(\Omega)$. If $|\Omega|$ is odd, we obtain part~\eqref{accent2}. Suppose then $n=|\Omega|$ is even. In particular, $\kappa=\ell-1\le 2$ and hence $\ell=3$. Clearly, $m_2$ is even and hence~$(\dag)$ applied with $i=1$ yields $s_1=2$. Thus $m_1=m_2^{s_1}=m_2^2\equiv 0\pmod 4$. Lemma~\ref{l:aschbacher}~\eqref{asch2} yields $G_1\le \Alt(\Omega)=G_3$, which is a contradiction.

Suppose that $\kappa=\ell-1$ and $G_\ell$ is a primitive group of HA type. If $|\Omega|$ is odd, we obtain part~\eqref{accent3}. Suppose then $n=|\Omega|$ is even, that is, $n=2^d$, for some positive integer $d\ge 3$. In particular, $\kappa=\ell-1\le 2$ and hence $\ell=3$. Clearly, $m_2$ is even and hence~$(\dag)$ applied with $i=1$ yields $m_2\equiv 2\pmod 4$. Therefore $m_2=2$, however this contradicts the fact that in a regular $(m,k)$-product struction we must have $m\ge 5$.

Finally suppose that $\kappa=\ell-2$, $G=\Sym(\Omega)$, $G_\ell=\Alt(\Omega)$ and $G_{\ell-1}$ is a primitive group of HA type. If $|\Omega|$ is even, then $|\Omega|=2^d$ for some $d\ge 3$. As $G_2\cong\mathrm{AGL}_d(2)\le \Alt(\Omega)=G_3$, we obtain a contradiction. Therefore $|\Omega|$ is odd and we obtain~\eqref{accent4}.
\end{proof}


\section{Boolean intervals containing a maximal imprimitive subgroup}\label{sec:impri1}
The scope of this section is to gather some information on Boolean lattices $\mathcal{O}_G(H)$ containing a maximal element that is imprimitive. Our main tool in this task is a result of Aschbacher and Shareshian~\cite[Theorem~$5.2$]{3}.
\begin{hypothesis}\label{hyp:51}{\em Let $G$ be either $\Sym(\Omega)$ or $\Alt(\Omega)$ with $n:=|\Omega|$, let $\Sigma$ be a non-trivial regular partition, let $G_1:=\nor G{\Sigma}$, let $G_2$ be a maximal subgroup of $G$ distinct from $\Alt(\Omega)$ and let $H:=G_1\cap G_2$. Assume that
\begin{itemize}
\item $\mathcal{O}_G(H)$ is a Boolean lattice of rank $2$ with maximal elements $M_1$ and $M_2$, and 
 \item $H$ acts transitively on $\Omega$.
\end{itemize}}
\end{hypothesis}

\begin{theorem}\label{thrm:AS}{{\cite[Theorem~$5.2$]{3}}}
Assume Hypothesis~$\ref{hyp:51}$. Then one of the following holds:
\begin{enumerate}
\item\label{enu1} For every $i\in \{1,2\}$, there exists a non-trivial regular partition $\Sigma_i$ with $G_i=\nor G{\Sigma_i}$; moreover, for some $i\in \{1,2\}$, $\Sigma_i< \Sigma_{3-i}$. Further, $n\ge 8$ and, if $n=8$, then $G=\Sym(\Omega)$.
\item\label{enu2}$G=\Alt(\Omega)$, $n=2^{a+1}$ for some positive integer $a>1$, $G_2$ is an affine primitive group, $V={\bf F}^*(G_2)\le H$, $V_\Sigma$ is a hyperplane of $V$, the elements of $\Sigma$ are the two orbits of $V_\Sigma$ on $\Omega$, and $H=\nor {G_2}{V_\Sigma}$. 
\item\label{enu3}$G=\Alt(\Omega)$, $n\equiv 0\pmod 4$, $n>8$ and, for every $i\in \{1,2\}$, there exists a non-trivial regular partition $\Sigma_i$ such that
\begin{description}
\item[(a)]$G_i=\nor G{\Sigma_i}$,
\item[(b)]$\Sigma_1$ and $\Sigma_2$ are lattice complements in the poset of partitions of $\Omega$, and
\item[(c)]one of $\Sigma_1$, $\Sigma_2$ is $(2,n/2)$-regular and the other is $(n/2,2)$-regular.
\end{description}
\end{enumerate}
\end{theorem}
(Observe that two partitions $\Sigma_1$ and $\Sigma_2$ of $\Omega$ are lattice complements if, the smallest partition $\Sigma$ of $\Omega$ with $\Sigma_1\le \Sigma$ and $\Sigma_2\le \Sigma$ and the largest partition $\Sigma'$ of $\Omega$ with $\Sigma'\le\Sigma_1$ and $\Sigma'\le \Sigma_2$ are the two trivial partitions of $\Omega$. Futher $V_\{\Sigma\}$ denote the poitwise stabilizer of $\Sigma$ in $V$)

\begin{hypothesis}\label{hyp:52}{\em Let $G$ be either $\Sym(\Omega)$ or $\Alt(\Omega)$ with $n:=|\Omega|$, let $\Sigma$ be a non-trivial regular partition, let $G_1:=\nor G{\Sigma}$, let $G_2$ and $G_3$ be maximal subgroups of $G$  and let $H:=G_1\cap G_2\cap G_3$. Assume that
\begin{itemize}
\item $\mathcal{O}_G(H)$ is a Boolean lattice of rank $3$ with maximal elements $G_1$, $G_2$ and $G_3$, and 
 \item $H$ acts transitively on $\Omega$.
\end{itemize}}
\end{hypothesis}

\begin{theorem}\label{thrm:AS1}
Assume Hypothesis~$\ref{hyp:52}$. Then one of the following holds:
\begin{enumerate}
\item\label{enu11} For every $i\in \{1,2,3\}$, there exists a non-trivial regular partition $\Sigma_i$ with $G_i=\nor G{\Sigma_i}$; moreover, relabeling the indexed set $\{1,2,3\}$ if necessary, $\Sigma_1< \Sigma_2< \Sigma_3$. 
\item\label{enu12}$G=\Sym(\Omega)$. Relabeling the indexed set $\{1,2,3\}$ if necessary, $G_3=\Alt(\Omega)$ and, for every $i\in \{1,2\}$, there exists a non-trivial regular partition $\Sigma_i$ with $G_i=\nor G{\Sigma_i}$, moreover, for some $i\in \{1,2\}$, $\Sigma_i< \Sigma_{3-i}$.
\item\label{enu13}$G=\Alt(\Omega)$, $|\Omega|=8$ and the Boolean lattice $\mathcal{O}_G(H)$ is in Figure~$\ref{fig1}$.
\end{enumerate}
\end{theorem}
\begin{proof}
If none of $G_1$, $G_2$ and $G_3$ is $\Alt(\Omega)$ and if $G=\Sym(\Omega)$, then the result follows directly from Theorem~\ref{thrm:AS} and we obtain~\eqref{enu11}.
Suppose $G=\Sym(\Omega)$ and one of $G_2$ or $G_3$ is $\Alt(\Omega)$. Without loss of generality we may assume that $G_3=\Alt(\Omega)$. Now, the result follows directly from Theorem~\ref{thrm:AS} applied to $\{G_1,G_2\}$; we obtain~\eqref{enu12}.

It remains to consider the case $G=\Alt(\Omega)$. In particular, we may apply Theorem~\ref{thrm:AS} to the pairs $\{G_1,G_2\}$ and $\{G_1,G_3\}$. Relabeling the indexed set $\{1,2,3\}$ if necessary, we have to consider in turn one of the following cases:
\begin{description}
\item[case A] Theorem~\ref{thrm:AS} part~\eqref{enu1} holds for both pairs $\{G_1,G_2\}$ and $\{G_1,G_3\}$;
\item[case B]Theorem~\ref{thrm:AS} part~\eqref{enu1} holds for $\{G_1,G_2\}$ and Theorem~\ref{thrm:AS} part~\eqref{enu2} holds for $\{G_1,G_3\}$;
\item[case C]Theorem~\ref{thrm:AS} part~\eqref{enu1} holds for $\{G_1,G_2\}$ and Theorem~\ref{thrm:AS} part~\eqref{enu3} holds for $\{G_1,G_3\}$;
\item[case D]Theorem~\ref{thrm:AS} part~\eqref{enu2} holds for both pairs $\{G_1,G_2\}$ and $\{G_1,G_3\}$;
\item[case E]Theorem~\ref{thrm:AS} part~\eqref{enu2} holds for $\{G_1,G_2\}$ and Theorem~\ref{thrm:AS} part~\eqref{enu3} holds for $\{G_1,G_3\}$;
\item[case F]Theorem~\ref{thrm:AS} part~\eqref{enu3} holds for both pairs $\{G_1,G_2\}$ and $\{G_1,G_3\}$. 
\end{description}

\smallskip

\noindent\textsc{Case A:} In particular, $G_2$ and $G_3$ are stabilizers of non-trivial regular partitions and hence we are in the position to apply Theorem~\ref{thrm:AS} also to the pair $\{G_2,G_3\}$. It is not hard to see that Theorem~\ref{thrm:AS} part~\eqref{enu1} holds for $\{G_2,G_3\}$ and that  the conclusion~\eqref{enu11} in the statement of Theorem~\ref{thrm:AS1} holds. 

\smallskip

\noindent\textsc{Case B: }Since $G_1$ is the stabilizer of a non-trivial regular partition and since $\{G_1,G_3\}$ satisfies Theorem~\ref{thrm:AS} part~\eqref{enu2}, we deduce that $G_3$ is an affine primitive group and $\Sigma_1$ is an $(n/2,2)$-regular partition.

Since $G_2$ is the stabilizer of the non-trivial regular partition $\Sigma_2$, we deduce that we may apply Theorem~\ref{thrm:AS} to the pair $\{G_2,G_3\}$. In particular, as $G_3$ is primitive, Theorem~\ref{thrm:AS} part~\eqref{enu2} must hold for $\{G_2,G_3\}$ and hence $G_2$ is the stabilizer of an $(n/2,2)$-regular partition. However, this contradicts the fact that $\{G_1,G_2\}$ satisfies Theorem~\ref{thrm:AS} part~\eqref{enu1}, that is, $\Sigma_1< \Sigma_2$ or $\Sigma_2< \Sigma_1$.

\smallskip

\noindent\textsc{Case C: }We have either
\begin{description}
\item[(a)]\label{aa} $\Sigma_1< \Sigma_2$, $\Sigma_1$ is a $(2,n/2)$-regular partition, $\Sigma_3$ is a $(n/2,2)$-regular partition and $\Sigma_1$, $\Sigma_3$ are lattice complements, or
\item[(b)]\label{bb} $\Sigma_2< \Sigma_1$, $\Sigma_1$ is a $(n/2,2)$-regular partition, $\Sigma_3$ is a $(2,n/2)$-regular partition and $\Sigma_1$, $\Sigma_3$ are lattice complements.
\end{description}
In case~{\bf (b)}, $\Sigma_2< \Sigma_1$ and hence $\Sigma_1$ is a refinement of $\Sigma_2$; however, as $\Sigma_1$ is a $(n/2,2)$-regular partition, this is not possible. Therefore, case~{\bf (b)} does not arise.
 As $G_2$ and $G_3$ are stabilizers of non-trivial regular partitions of $\Omega$, we are in the position to apply Theorem~\ref{thrm:AS} also to the pair $\{G_2,G_3\}$. If Theorem~\ref{thrm:AS} part~\eqref{enu1} holds for $\{G_2,G_3\}$, then either $\Sigma_2< \Sigma_3$ or $\Sigma_3< \Sigma_2$. However, both possibilities lead to a contradiction. Indeed, if $\Sigma_2< \Sigma_3$ and {\bf (a)} holds, then $\Sigma_1< \Sigma_2< \Sigma_3$, contradicting the fact that $\Sigma_1$ and $\Sigma_3$ are lattice complements. The argument when $\Sigma_3< \Sigma_2$ is analogous. Similarly, if Theorem~\ref{thrm:AS} part~\eqref{enu3} holds for $\{G_2,G_3\}$, then $\Sigma_2$ and $\Sigma_3$ are lattice complements and either
\begin{description}
\item[(a)']$\Sigma_2$ is a $(2,n/2)$-regular partition and $\Sigma_3$ is a $(n/2,2)$-regular partition, or
\item[(b)']$\Sigma_2$ is a $(n/2,2)$-regular partition and $\Sigma_3$ is a $(2,n/2)$-regular partition.
\end{description} 
However, an easy case-by-case analysis shows that {\bf (a)'} and {\bf (b)'} are incompatible with {\bf (a)}. 

\smallskip

\noindent\textsc{Case D: }In particular, $G_2$ and $G_3$ are both primitive groups of affine type. Let $V_2$ be the socle of $G_2$ and let $V_3$ be the socle of $G_3$. From Lemma~\ref{l:1} applied to  $\mathcal{O}_{G}(G_2\cap G_3)$, we deduce that either $G_2\cap G_3$ is primitive, or $G=\Alt(\Omega)$, $|\Omega|=8$ and $G_2\cap G_3$ is the stabilizer of a $(2,4)$-regular partition. In the latter case, we see with a direct computation that part~\eqref{enu13} holds. Suppose then that $G_2\cap G_3$ is primitive.   From Lemma~\ref{l:orcoboiaboia} applied to the inclusions $G_2\cap G_3<G_2$ and $G_2\cap G_3<G_3$, we deduce that either 
\begin{description}
\item[(a)'']$G_2\cap G_3$, $G_2$ and $G_3$ have the same socle, or
\item[(b)''] $n=8$, $G_2\cap G_3\cong \mathrm{PSL}_2(7)$ and $G_2\cong G_3\cong \mathrm{AGL}_3(2)$.
\end{description} In the former case, we have $V_2=V_3$ and hence $G_2=\nor G {V_2}=\nor G {V_3}=G_3$, contradicting the fact that $G_2\ne G_3$. In the latter case, we have checked with the invaluable help of the computer algebra system \texttt{magma}~\cite{magma} that $\mathcal{O}_{\Alt(8)}(\mathrm{PSL}_2(7))=\{\mathrm{PSL}_2(7)<\mathrm{AGL}_3(2)<\Alt(8)\}$, contradicting the fact that it is a Boolean lattice.

\smallskip

\noindent\textsc{Case E: }In this case, $\Sigma_1$ is a $(n/2,2)$-regular partition, $\Sigma_3$ is a $(2,n/2)$-regular partition and $\Sigma_1$, $\Sigma_3$ are lattice complements. As $G_3$ is the stabilizer of a non-trivial regular partition, we are in the position to apply Theorem~\ref{thrm:AS} to the pair $\{G_2,G_3\}$. As $G_2$ is primitive, we see that Theorem~\ref{thrm:AS} part~\eqref{enu2} holds for $\{G_2,G_3\}$ and hence $\Sigma_3$ is a $(n/2,2)$-regular partition, which implies $(n/2,2)=(2,n/2)$, that is, $n=4$. However this contradicts $a>1$ in Theorem~\ref{thrm:AS} part~\eqref{enu2}. 

\smallskip

\noindent\textsc{Case F: }In particular, both $\Sigma_2$ and $\Sigma_3$ are either  $(n/2,2)$-regular partitions or $(2,n/2)$-regular partitions. As $G_2$ and $G_3$ are stabilizers of non-trivial regular partitions, we may apply Theorem~\ref{thrm:AS} also to the pair $\{G_2,G_3\}$. Clearly, none of parts~\eqref{enu1},~\eqref{enu2} and~\eqref{enu3} in Theorem~\ref{thrm:AS} holds for $\{G_2,G_3\}$, which is a contradiction. 
\end{proof}

\begin{corollary}\label{cor:111}
Let $H$ be a transitive subgroup of $G$ and  suppose that $\mathcal{O}_G(H)$ is Boolean of rank $\ell\ge 3$ and that $\mathcal{O}_G(H)$ contains a maximal element which is imprimitive. Let $\{G_1,\ldots,G_\ell\}$ be the maximal elements of $\mathcal{O}_G(H)$. Then one of the following holds:
\begin{enumerate}
\item\label{enu11a} For every $i\in \{1,\ldots,\ell\}$, there exists a non-trivial regular partition $\Sigma_i$ with $G_i=\nor G{\Sigma_i}$; moreover, relabeling the indexed set $\{1,\ldots,\ell\}$ if necessary, $\Sigma_1< \cdots < \Sigma_\ell$. 
\item\label{enu12a}$G=\Sym(\Omega)$. Relabeling the indexed set $\{1,\ldots,\ell\}$ if necessary, $G_\ell=\Alt(\Omega)$, for every $i\in \{1,\ldots,\ell-1\}$, there exists a non-trivial regular partition $\Sigma_i$ with $G_i=\nor G{\Sigma_i}$, and  $\Sigma_1< \cdots<  \Sigma_{\ell-1}$.
\item\label{enu13a}$G=\Alt(\Omega)$, $|\Omega|=8$, $\ell=3$ and the Boolean lattice $\mathcal{O}_G(H)$ is in Figure~$\ref{fig1}$.
\end{enumerate}
\end{corollary}
\begin{proof}
It follows arguing inductively on $\ell$; the base case $\ell=3$ is Theorem~\ref{thrm:AS1}.
\end{proof}


\section{Boolean intervals containing a maximal intransitive subgroup}
The scope of this section is to gather some information on Boolean lattices $\mathcal{O}_G(H)$ containing a maximal element that is intransitive. Some of the material in this section can be also traced back to the PhD thesis~\cite{basile}.
\begin{hypothesis}\label{hyp:53}{\em Let $G$ be either $\Sym(\Omega)$ or $\Alt(\Omega)$ with $n:=|\Omega|$, let $\Gamma$ be a subset of $\Omega$ with $1\le |\Gamma|<|\Omega|/2$, let $G_1:=\nor G{\Gamma}$, let $G_2$ be a maximal subgroup of $G$ and let $H:=G_1\cap G_2$. Assume that $\mathcal{O}_G(H)$ is Boolean of rank $2$ with maximal elements $G_1$ and $G_2$.
}
\end{hypothesis}

\begin{theorem}\label{thrm:AS2}
Assume Hypothesis~$\ref{hyp:53}$. Then one of the following holds:
\begin{enumerate}
\item\label{enu111}$G=\Sym(\Omega)$ and $G_2=\Alt(\Omega)$. 
\item\label{enu222}$G_2$ is an imprimitive subgroup having $\Gamma$ as a block of imprimitivity. 
\item\label{enu222HELL}$G=\Alt(\Omega)$, $n=7$, $|\Gamma|=3$ and $G_2\cong\mathrm{SL}_3(2)$ acts primitively on $\Omega$.
\item\label{enu333}$|\Gamma|=1$ and one of the following holds:
\begin{description}
\item[(a)] $G=\Alt(\Omega)$, $G_2\cong\mathrm{AGL}_d(2)$ with $d\ge 3$;
\item[(b)] $G=\Alt(\Omega)$, $G_2\cong \mathrm{Sp}_{2m}(2)$ and $|\Omega|\in \{2^{m-1}(2^m+1),2^{m-1}(2^m-1)\}$; 
\item[(c)]$G=\Alt(\Omega)$, $G_2\cong HS$ and $|\Omega|=176$;
\item[(d)]$G=\Alt(\Omega)$, $G_2\cong Co_3$ and $|\Omega|=276$;
\item[(e)]$G=\Alt(\Omega)$, $G_2\cong M_{12}$ and $|\Omega|=12$;
\item[(f)]$G=\Alt(\Omega)$, $G_2\cong M_{24}$ and $|\Omega|=24$;
\item[(g)]$G=\Sym(\Omega)$, $G_2\cong \mathrm{PGL}_2(p)$ with $p$ prime and $|\Omega|=p+1$;
\item[(h)]$G=\Alt(\Omega)$, $G_2\cong \mathrm{PSL}_2(p)$ with $p$ prime and $|\Omega|=p+1$.  
\end{description}
\end{enumerate}
\end{theorem}
\begin{proof}

Suppose that $G_2$ is intransitive. Thus $G_2=G\cap (\Sym(\Gamma')\times \Sym(\Omega\setminus \Gamma'))$, for some subset $\Gamma'\subseteq \Omega$ with $1\le |\Gamma'|<|\Omega|/2$. In particular, 
$$H=G_1\cap G_2=G\cap (\Sym(\Gamma\cap \Gamma')\times \Sym(\Gamma\setminus \Gamma')\times \Sym(\Gamma'\setminus\Gamma)\times \Sym(\Omega\cup (\Gamma\cup \Gamma'))).$$ Thus $H$ is contained in 
\begin{itemize}
\item$G\cap (\Sym(\Gamma\cap \Gamma')\times \Sym(\Omega\setminus(\Gamma\cap\Gamma')))$, 
\item $G\cap (\Sym(\Gamma\setminus \Gamma')\times \Sym(\Omega\setminus(\Gamma\setminus\Gamma')))$,
\item$G\cap (\Sym(\Gamma'\setminus \Gamma)\times \Sym(\Omega\setminus(\Gamma'\setminus\Gamma)))$,
\item $G\cap (\Sym(\Gamma\cup \Gamma')\times \Sym(\Omega\setminus(\Gamma\cup\Gamma')))$.   
\end{itemize}
Since the only overgroups of $H$ are $H,G_1,G_2$ and $G$, each of the previous four subgroups must be one of 
$H,G_1,G_2$ and $G$. This immediately implies $G=G\cap (\Sym(\Gamma\cap \Gamma')\times \Sym(\Omega\setminus(\Gamma\cap\Gamma')))$, that is, $\Gamma\cap \Gamma'=\emptyset$. However, $G\cap (\Sym(\Gamma\cup \Gamma')\times \Sym(\Omega\setminus(\Gamma\cup\Gamma')))$   is neither $H$, nor $G_1$, nor $G_2$, nor $G$, because $1\le |\Gamma|,|\Gamma'|<|\Omega|/2$.

\smallskip

Suppose that $G_2$ is imprimitive. 
In particular, $G_2$ is the stabilizer of a non-trivial $(a,b)$-regular partition of $\Omega$, that is, $G_2$ is the stabilizer of  a partition $\Sigma_2:=\{X_1,\ldots,X_b\}$ of the set $\Omega$ into $b$ parts each having cardinality $a$, for some positive integers $a$ and $b$ with $a,b\ge 2$. Thus $$G_2\cong G\cap(\Sym(a)\mathrm{wr}\Sym(b)).$$ The group $H=G_1\cap G_2$ is intransitive. Since $G_1$ is the only proper overgroup of $H$ that is intransitive, we deduce that $H$ has only two orbits on $\Omega$, namely $\Gamma$ and $\Omega\setminus \Gamma$. From this it follows that, for every $i\in \{1,\ldots,b\}$, either $X_i\subseteq \Gamma$ or $X_i\subseteq \Omega\setminus \Gamma$. Therefore the group $H=G_1\cap G_2$ is isomorphic to 
$$G\cap (\Sym(a)\wr\Sym(b_1)\times \Sym(a)\wr\Sym(b_2)),$$
where $b_1$ is the number of parts in $\Sigma_2$ contained in $\Gamma$ and $b_2$ is the number of parts in $\Sigma_2$ contained in $\Omega\setminus \Gamma$. Therefore, $H$ is contained in subgroups isomorphic to  
$$(\dag)\qquad G\cap (\Sym(ab_1)\times \Sym(a)\wr\Sym(b_2))\quad\textrm{and}\quad G\cap(\Sym(a)\wr\Sym(b_1)\times \Sym(ab_2)).$$
Since $H$ and $G_1$ are the only intransitive overgroup of $H$, we deduce that the two subgroups in~$(\dag)$ are $H$ or $G_1$. However this happens if and only if $b_1=1$. In other words, this happens if and only if $\Gamma\in\Sigma_2$ and we obtain part~\eqref{enu222}.

\smallskip

Suppose that $G_2$ is primitive. We divide our analysis in various cases.

\smallskip

\noindent\textsc{Case 1:} $|\Gamma|\ge 3$, or $|\Gamma|=2$ and $G=\Sym(\Omega)$.

\smallskip

\noindent Now, $H=G_1\cap G_2$ is a maximal subgroup of $G_1$. Moreover, $G_1=\Sym(\Gamma)\times \Sym(\Omega\setminus \Gamma)$ when $G=\Sym(\Omega)$ and $G_1=\Alt(\Omega)\cap (\Sym(\Gamma)\times \Sym(\Omega\setminus \Gamma))$ when $G=\Alt(\Omega)$. Consider $\pi_a:G_1\to \Sym(\Gamma)$ and $\pi_b:G_1\to \Sym(\Omega\setminus\Gamma)$ the natural projections.  
Oberve that these projections are surjective.

Assume $\pi_a(G_1\cap G_2)$ is a proper subgroup of $\Sym(\Gamma)$. Then, from the maximality of $G_1\cap G_2$ in $G_1$, we have
$$G_1\cap G_2=G\cap (\pi_a(G_1\cap G_2)\times \Sym(\Omega\setminus\Gamma)).$$ As $|\Omega\setminus\Gamma|\ge 3$, we deduce that $G_1\cap G_2$ contains a $2$-cycle or a $3$-cycle. In particular, the primitive group $G_2$ contains a $2$-cycle or a $3$-cycle. By a celebrated result of Jordan~\cite[Theorem~$3.3$A]{DM}, we  obtain $\Alt(\Omega)\le G_2$. Thus $G=\Sym(\Omega)$ and $G_2=\Alt(\Omega)$ and we obtain part~\eqref{enu111}. Suppose then $\pi_a(G_1\cap G_2)=\Sym(\Gamma)$ and let $K_a:=\mathrm{Ker}(\pi_a)\cap G_1\cap G_2$.

If $\pi_b(G_1\cap G_2)$ is a proper subgroup of $\Sym(\Omega\setminus\Gamma)$, using the same argument of the previous paragraph we obtain part~\eqref{enu111}.


Suppose then $\pi_b(G_1\cap G_2)=\Sym(\Omega\setminus\Gamma)$ and let $K_b:=\mathrm{Ker}(\pi_b)\cap G_1\cap G_2$. In the rest of the proof of this case the reader might find useful to see Figure~\ref{fig2}.

\begin{figure}
\begin{tikzpicture}[node distance=2cm]
\node(A0){$H=G_1\cap G_2$};
\node(A1)[below of=A0]{$K_aK_b$};
\node(A2)[below of=A1]{$K_a$};
\node(A3)[right of=A2]{$K_b$};
\node(A4)[below of=A3]{$K_a\cap K_b=1$};
\draw(A0)--(A1);
\draw(A1)--(A2);
\draw(A1)--(A3);
\draw(A2)--(A4);
\draw(A3)--(A4);
\path[-] (A0) edge [in=120, out=230] node [left, rotate=0]{$\Sym(\Gamma)\cong$}(A2);
\path[-] (A0) edge [in=100, out=-50] node [right, rotate=0]{$\cong\Sym(\Omega\setminus\Gamma)$}(A3);
\end{tikzpicture}
\caption{Structure of $H=G_1\cap G_2$}\label{fig2}
\end{figure}
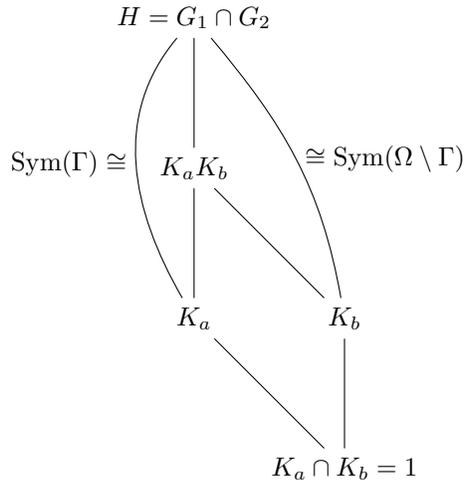

Now, $K_aK_b$ is a subgroup of $G_1\cap G_2$, moreover $(G_1\cap G_2)/(K_aK_b)$ is an epimorphic image of both $\Sym(\Gamma)$ and $\Sym(\Omega\setminus \Gamma)$. Assume $|\Omega\setminus \Gamma|\ge 5$. Then the only epimorphic image of both $\Sym(\Gamma)$ and $\Sym(\Omega\setminus \Gamma)$ is either the identity group or the cyclic group of order $2$. Therefore $|G_1\cap G_2:K_aK_b|\le 2$. Moreover, $K_aK_b/K_b\cong K_a/(K_a\cap K_b)=K_a$ is isomorphic to either $\Alt(\Omega\setminus\Gamma)$ or to $\Sym(\Omega\setminus \Gamma)$. In both cases, $\Alt(\Omega\setminus\Gamma)\le K_a\le G_2$ and hence $G_2$ contains a $3$-cycle. As above, this implies $G=\Sym(\Omega)$ and $G_2=\Alt(\Omega)$ and part~\eqref{enu111} holds. Assume $|\Omega\setminus \Gamma|\le 4$. As $1\le |\Gamma|<|\Omega|/2$, we deduce $|\Omega|\le 7$. When $|\Gamma|=3$, we obtain $|\Omega|=7$ and we can verify with a direct analysis that part~\eqref{enu1111} holds when $G=\Sym(\Omega)$ and part~\eqref{enu222HELL} holds when $G=\Alt(\Omega)$. Finally, if $|\Gamma|=2$, we have $|\Omega|\in \{5,6\}$ and $G=\Sym(\Omega)$. A direct inspection in each of these cases reveals that every maximal subgroup of $G_1$ contains either a $2$-cycle or a $3$-cycles. Therefore $G_2=\Alt(\Omega)$ and part~\eqref{enu111} holds.

\smallskip

\noindent\textsc{Case 2:} $|\Gamma|=2$ and $G=\Alt(\Omega)$.

\smallskip

\noindent In this case, $G_1=\Alt(\Omega)\cap (\Sym(\Gamma)\times \Sym(\Omega\setminus\Gamma))\cong \Sym(\Omega\setminus\Gamma)$.

Assume that $H=G_1\cap G_2$ acts intransitively on $\Omega\setminus\Gamma$ and let $\Delta$ be one of its smallest orbits. In particular, $H$ fixes setwise $\Gamma$, $\Delta$ and $\Omega\setminus(\Gamma\cup\Delta)$. Now, $\Alt(\Omega)\cap (\Sym(\Gamma\cup\Delta)\times\Sym(\Omega\setminus(\Gamma\cup \Delta)))$ is a proper overgroup of $H$ which is intransitive and is different from $G_1$, which is a contradiction. Therefore $H$ acts transitively on $\Omega\setminus \Gamma$. Suppose that $H$ acts imprimitively on $\Omega\setminus\Gamma$. Since $H$ is maximal in $G_1\cong\Sym(\Omega\setminus \Gamma)$, we deduce $H=\nor {G_1}\Sigma$, where $\Sigma$ is a non-trivial $(a,b)$-regular partition of $\Omega\setminus \Gamma$. If $a\ge 3$, then $H$ contains a $3$-cycle and hence so does $G_2$. Since $G_2$ is primitive, we deduce from~\cite[Theorem~$3.3$A]{DM} that $G_2=\Alt(\Omega)=G$, which is a contradiction. 
If $a=2$, then $H$ contains a permutation that is the product of two disjoint transpositions. Since $G_2$ is primitive, we deduce from~\cite[Theorem~$3.3$D and Example~$3.3.1$]{DM} that either $G_2=\Alt(\Omega)=G$ or $|\Omega|\le 8$. The first possibility is clearly impossible and hence $|\Omega|\in \{6,8\}$. However, a computation in $\Alt(6)$ and in $\Alt(8)$ reveals that no case arises.  Therefore $H$ acts primitively on $\Omega\setminus \Gamma$.
 
Let $\Gamma=\{\gamma,\gamma'\}$. As $|\Gamma|=2$, the group $(G_1\cap G_2)_\gamma=H_\gamma$ has index at most $2$ in $G_1\cap G_2=H$ and hence $H_\gamma\unlhd H$. Since $H$ acts primitively on $\Omega\setminus\Gamma$ and $H_\gamma\unlhd H$, $H_\gamma$ acts transitively on  $\Omega\setminus \Gamma$ or $H_\gamma$ is trivial. The second possibility is clearly a contradiction because it implies $|H|= 2$ and hence $|\Omega|=4$. Thus $H_\gamma$ acts transitively on $\Omega\setminus \Gamma$ and the orbits of $H_\gamma$ on $\Omega$ are $\{\gamma\},\{\gamma'\},\Omega\setminus \Gamma$ and have cardinality $1,1,|\Omega|-2$. Since $G_2$ is primitive and not regular, from Lemma~\ref{l:4}, we deduce that $\gamma$ is the only fixed point of $(G_2)_\gamma$.
 Since $H_\gamma$ is a subgroup of $(G_2)_\gamma$ from the cardinality of the orbits of $H_\gamma$, we deduce that 
$(G_2)_\gamma$ acts transitively on $\Omega\setminus\{\gamma\}$, that is, $G_2$ is $2$-transitive. Similarly, since $H_\gamma\le (G_2)_{\gamma}\cap (G_2)_{\gamma'}$, we deduce also that $G_2$ is $3$-transitive.  

From the classification of the finite $3$-transitive groups, we deduce that 
\begin{enumerate}
\item $G_2$ equals the Mathieu group $M_{n}$ and $n=|\Omega|\in \{11,12,22,23,24\}$, or
\item $G_2=M_{11}$ and $|\Omega|=12$, or
\item ${\bf F}^*(G_2)=\mathrm{PSL}_2(q)$ and $|\Omega|=q+1$.
\end{enumerate}
Using this information, a computation with the computer algebra system \texttt{magma} shows that the cases~(1) and~(2) do not arise because $\mathcal{O}_G(H)$ is not Boolean of rank $2$. In case~(3), from the structure of $\mathrm{PSL}_2(q)$, we deduce that $G_1\cap G_2$ is solvable and hence $G_1\cap G_2$ is a solvable group acting primitively on $|\Omega|-2$ points. This yields that $q-1$ is a prime power, say $q-1=x^y$, for some prime $x$ and for some positive integer $y$. Write $q=p^f$, for some prime power $p$ and some positive integer $f$. Since $p^f-1$ is a power of a prime, we deduce that $p^f-1$ has no primitive prime divisors. From a famous result of Zsigmondy~\cite{zi}, this yields 
\begin{description}
\item[(a)]$f=1$, $x=2$ and $q-1=2^y$, or
\item[(b)] $q=9$, $x=2$ and $y=3$ or
\item[(c)] $p=2$, $f$ is prime and $q-1=2^f-1=x$ is a prime.
\end{description} We can now refine further our argument above. Indeed, recall that $G_1\cap G_2$ is a maximal subgroup of $G_1\cong \Sym(\Omega\setminus\Gamma)$. Since $G_1\cap G_2$ is solvable, we deduce that $G_1\cap G_2$ is isomorphic to the general linear group $\mathrm{AGL}_y(x)$ and hence $|G_1\cap G_2|=x^y|\mathrm{GL}_y(x)|=(q-1)|\mathrm{GL}_y(x)|$. Since $G_2=\nor {\Alt(q+1)}{\mathrm{PSL}_2(q)}$ and $|\Aut(\mathrm{PSL}_2(q))|=fq(q^2-1)$, we deduce that $|G_1\cap G_2|$ divides $2f(q-1)$. Therefore $|\mathrm{GL}_y(x)|$ divides $2f$. Cases~{\bf (a)} and~{\bf (b)} are readily seen to be impossible and in case~{\bf (c)} we have $|\mathrm{GL}_1(x)|=2^f-2=2(2^{f-1}-1)$ divides $2f$, which is possible only when $f=3$. A computation reveals that in this latter case $\mathcal{O}_G(H)$ has 5 elements and hence it is not Boolean.

\smallskip

\noindent\textsc{Case 3:} $|\Gamma|=1$.

\smallskip

\noindent We assume that the conclusion in part~\eqref{enu111} of this lemma does not hold and hence $G_2$ is a primitive subgroup of $G$ with $\Alt(\Omega)\nleq G_2$.

Assume that $H=G_1\cap G_2$ acts intransitively on $\Omega\setminus\Gamma$ and let $\Delta$ be one of its smallest orbits. In particular, $H$ fixes setwise $\Gamma$, $\Delta$ and $\Omega\setminus(\Gamma\cup\Delta)$. Now, $\Alt(\Omega)\cap (\Sym(\Gamma\cup\Delta)\times\Sym(\Omega\setminus(\Gamma\cup \Delta)))$ is a proper overgroup of $H$ which is intransitive and is different from $G_1$, which is a contradiction. Therefore $H$ acts transitively on $\Omega\setminus \Gamma$. Suppose that $H$ acts imprimitively on $\Omega\setminus\Gamma$. Since $H$ is maximal in $G_1\cong\Sym(\Omega\setminus \Gamma)$, we deduce $H=\nor {G_1}\Sigma$, where $\Sigma$ is a non-trivial $(a,b)$-regular partition of $\Omega\setminus \Gamma$. If $a\ge 3$, then $H$ contains a $3$-cycle and hence so does $G_2$. Since $G_2$ is primitive, we deduce from~\cite[Theorem~$3.3$A]{DM} that $\Alt(\Omega)\le G_2$, which is a contradiction. If $a=2$, then $H$ contains a permutation that is the product of two disjoint transpositions. Since $G_2$ is primitive, we deduce from~\cite[Theorem~$3.3$D and Example~$3.3.1$]{DM} that either $\Alt(\Omega)\le G_2$ or $|\Omega|\le 8$. The first possibility is clearly impossible. In the second case, as $a=2$, we have that $|\Omega\setminus \Gamma|$ is even and hence $|\Omega|\in \{5,7\}$. However, a computation in $\Alt(5)$, $\Sym(5)$, $\Alt(7)$  and $\Sym(7)$ reveals that no case arises.  Therefore $$H \textrm{ acts primitively on }\Omega\setminus \Gamma.$$

In particular, $G_2$ is $2$-transitive on $\Omega$. One of the first main applications on the Classification of the Finite Simple Groups is the classification of the finite $2$-transitive groups, see~\cite{peter}. These groups are either affine or almost simple. For the rest of the proof we go through this classification for  investigating $G_2$ further; we assume that the reader is broadly familiar with this classification and  for this part we refer the reader to Section~$7.7$ in~\cite{DM}.

\smallskip

\noindent\textsc{Case 3A:} $G_2$ is affine.

\smallskip

\noindent Since $G_2$ is a maximal subgroup of $G$, we deduce that $G_2\cong G\cap \mathrm{AGL}_d(p)$, for some prime number $p$ and some positive integer $d$. Now, $G_1\cap G_2\cong G\cap\mathrm{GL}_d(p)$ and the action of $G_1\cap G_2$ on $\Omega\setminus\Gamma$ is permutation isomorphic to the natural action of a certain subgroup of index at most $2$ of the linear group $\mathrm{GL}_d(p)$ acting  on the non-zero vectors of a $d$-dimensional vector space over the field with $p$-elements. Clearly, this action is primitive if and only if $d=1$ and $p-1$ is prime, or $p=2$. Indeed, if $V$ is the $d$-dimensional vector space over the field $\mathbb{F}_p$  with $p$ elements, then $\mathrm{GL}_d(p)$ preserves the partition $\{\{av\mid a\in \mathbb{F}_p,a\ne 0\}\mid v\in V,v\ne 0\}$ of $V\setminus\{0\}$. This partition is the trivial partition only when $p=2$ or $d=1$.  When $d=1$, the group $\mathrm{GL}_1(p)$ is cyclic of order $p-1$ and it acts primitively on $V\setminus\{0\}$ if and only if $p-1$ is a prime number. Since the only two consecutive primes are $2$ and $3$, in the latter case we obtain $|\Omega|=3$ and no case arises here. Thus $p=2$.

 If $d\le 2$, then $\Alt(\Omega)\le G_2$, which is a contradiction. Therefore $d\ge 3$. With a computation (using the fact that $\mathrm{GL}_d(2)$ is generated by transvectoins for example) we see that, when $d\ge 3$, the group $\mathrm{AGL}_d(2)$ consists of even permutations and hence $\mathrm{AGL}_d(2)\le \Alt(\Omega)$. This implies $G=\Alt(\Omega)$ and we obtain one of the examples stated in the theorem, namely part~\eqref{enu333}~{\bf (a)}.

\smallskip

\noindent\textsc{Case 3B:} $G_2\cong\mathrm{Sp}_{2m}(2)$ and $|\Omega|=2^{m-1}(2^m+1)$ or $|\Omega|=2^{m-1}(2^m-1)$. 

\smallskip

\noindent The group $G_1\cap G_2$ is isomorphic to either $\mathrm{O}_{2m}^+(2)$ or to $\mathrm{O}_{2m}^-(2)$ depending on whether $|\Omega|=2^{m-1}(2^m+1)$ or $|\Omega|=2^{m-1}(2^m-1)$. Since $G_2$ is a simple group, we deduce $G_2\le\Alt(\Omega)$ and hence $G=\Alt(\Omega)$. We obtain part~\eqref{enu333}~{\bf (b)}.

\smallskip

\noindent\textsc{Case 3C:} ${\bf F}^*(G_2)\cong\mathrm{PSU}_{3}(q)$ and $|\Omega|=q^3+1$.

\smallskip

\noindent Let $q=p^f$, for some prime number $p$ and for some positive integer $f$. Observe that $G_1\cap G_2$ is solvable,  it is a maximal subgroup of $G_1$ and  it acts primitively on $\Omega\setminus\Gamma$. From this  we deduce that $G_1\cap G_2$ is isomorphic to $G\cap \mathrm{AGL}_{3f}(p)$. Since $|\Aut(\mathrm{PSU}_3(q))|=2f(q^3+1)q^3(q^2-1)$ and $|\Omega|=q^3+1$, we deduce that $G_1\cap G_2$ has order a divisor of $2fq^3(q^2-1)$. Therefore $|\mathrm{AGL}_{3f}(p)|=q^3|\mathrm{GL}_{3f}(p)|$ divides $4fq^3(q^2-1)$ (observe that the extra ``2'' in front of $2fq^3(q^2-1)$ takes in account the case that $G=\Alt(\Omega)$ and $G\cap\mathrm{AGL}_{3f}(p)$ has index $2$ in $\mathrm{AGL}_{3f}(p)$). Therefore $|\mathrm{GL}_{3f}(p)|$ divides $4f(q^2-1)$. However the inequality $|\mathrm{GL}_{3f}(p)|\le 4f(p^{2f}-1)$ is never satisfied.

\smallskip

\noindent\textsc{Case 3D:} ${\bf F}^*(G_2)\cong\mathrm{Sz}(q)$, $q=2^f$ for some odd positive integer $f\ge 3$ and $|\Omega|=q^2+1$.

\smallskip

\noindent  Since $\Aut(\mathrm{Sz}(q))\cong \mathrm{Sz}(q).f$ and since $f$ is odd, we deduce $G_2\le \Alt(\Omega)$. In particular, $G=\Alt(\Omega)$. As in the case above, $G_1\cap G_2$ is solvable,  $G_1\cap G_2$ is a maximal subgroup of $G_1$ and  $G_1\cap G_2$ acts primitively on $\Omega\setminus\Gamma$. From this  we deduce that $G_1\cap G_2$ is isomorphic to $G\cap\mathrm{AGL}_{2f}(2)$. Since $|\Aut(\mathrm{Sz}(q))|=f(q^2+1)q^2(q-1)$ and $|\Omega|=q^2+1$, we deduce that $G_1\cap G_2$ has order a divisor of $fq^2(q-1)$. Therefore $|\mathrm{AGL}_{2f}(2)|=q^2|\mathrm{GL}_{2f}(2)|$ divides $4fq^2(q-1)$. Therefore $|\mathrm{GL}_{2f}(2)|$ divides $4f(q-1)$. However the inequality $|\mathrm{GL}_{2f}(2)|\le 4f(2^{f}-1)$ is never satisfied.

\smallskip

\noindent\textsc{Case 3E:} ${\bf F}^*(G_2)\cong\mathrm{Ree}(q)$, $q=3^f$ for some odd positive integer $f\ge 1$ and $|\Omega|=q^3+1$.

\smallskip

\noindent  Since $\Aut(\mathrm{Ree}(q))\cong \mathrm{Ree}(q).f$ and since $f$ is odd, we deduce $G_2\le \Alt(\Omega)$. In particular, $G=\Alt(\Omega)$. As in the case above, $G_1\cap G_2$ is solvable,  $G_1\cap G_2$ is a maximal subgroup of $G_1$ and  $G_1\cap G_2$ acts primitively on $\Omega\setminus\Gamma$. From this  we deduce that $G_1\cap G_2$ is isomorphic to $G\cap\mathrm{AGL}_{3f}(3)$. Since $|\Aut(\mathrm{Ree}(q))|=f(q^3+1)q^3(q-1)$ and $|\Omega|=q^3+1$, we deduce that $G_1\cap G_2$ has order a divisor of $fq^3(q-1)$. Therefore $|\mathrm{AGL}_{3f}(3)|=q^3|\mathrm{GL}_{3f}(3)|$ divides $4fq^3(q-1)$. Therefore $|\mathrm{GL}_{3f}(3)|$ divides $4f(q-1)$. However the inequality $|\mathrm{GL}_{3f}(3)|\le 4f(3^{f}-1)$ is never satisfied.

\smallskip

\noindent\textsc{Case 3F:} $(G_2,|\Omega|)\in \{(HS,176),(Co_3,276),(\Alt(7),15),(\mathrm{PSL}_2(11),11),(M_{11},12)\}$.

\smallskip

\noindent  Since $\mathrm{PSL}_2(11)<M_{11}$ in their degree $11$ actions, $\Alt(7)<\mathrm{PSL}_4(2)$ in their degree $15$ actions and $M_{11}<M_{12}$ in their degree $12$ actions, we see that $\mathrm{PSL}_2(11)$, $\Alt(7)$ and $M_{12}$ are not maximal in $G$ and hence cannot be $G_2$. Therefore, we are left with $(G_2,|\Omega|)\in \{(HS,176),(Co_3,276)\}$. We obtain part~\eqref{enu333}~{\bf (c)} and~{\bf (d)}.

\smallskip

\noindent\textsc{Case 3G:} $(G_2,|\Omega|)\in \{(M_{11},11),(M_{12},12),(M_{22},22),(M_{22}.2,22),(M_{23},23),(M_{24},24)\}$.

\smallskip

\noindent With a computer computation we see that when $G_2\cong M_{11}$ the lattice $\mathcal{O}_{G}(H)$ is not Boolean. The cases $M_{22}$ and $M_{22}.2$ do not arise because in these two cases $G_1\cap G_2$ is isomorphic to either $\mathrm{PSL}_3(4)$ (when $G=\Alt(\Omega)$) or to $\mathrm{P}\Sigma\mathrm{L}_3(4)$ (when $G=\Sym(\Omega)$). However, these two groups are not maximal subgroups of $G_1$ because they are contained respectively in $\mathrm{PGL}_3(4)$ and in $\mathrm{P}\Gamma\mathrm{L}_3(4)$. Therefore, we are left with $(G_2,|\Omega|)\in \{(M_{12},12),(M_{23},23),(M_{24},24)\}$. The case $(G_2,|\Omega|)=(M_{23},23)$ also does not arise because with a computation computation we see that $\mathcal{O}_G(H)$ consists of five elements. Thus we are only left with part~\eqref{enu333}~{\bf (e)} and~{\bf (f)}.

\smallskip

\noindent\textsc{Case 3I:} ${\bf F}^*(G_2)\cong \mathrm{PSL}_d(q)$ for some prime power $q$ and some positive integer $d\ge 2$ and $|\Omega|=(q^d-1)/(q-1)$.

\smallskip

\noindent Since the group $G_2$ is acting on the points of a $(d-1)$-dimensional projective space, we deduce that $G_1\cap G_2$ acts primitively on $\Omega\setminus\Gamma$ only when $G_2$ is acting on the projective line, that is, $d=2$. (Indeed, consider the action of $X:=\mathrm{P}\Gamma\mathrm{L}_d(q)$  on the points of the projective space $\mathcal{P}$, consider a point $p$ of $\mathcal{P}$ and consider the stabilizer $Y$ of the point $p$ in $X$. Then $Y$ preserves a natural partition on $\mathcal{P}\setminus\{p\}$, where two points $p_1$ and $p_2$ are declared to be in the same part if the lines $\langle p,p_1\rangle$ and $\langle p,p_2\rangle$ are equal. This partition is trivial only when $\mathcal{P}$ is a line, that is, $d=2$.) Let $q=p^f$, for some prime number $p$ and for some positive integer $f$. Observe that $G_1\cap G_2$ is solvable,  it is a maximal subgroup of $G_1$ and  it acts primitively on $\Omega\setminus\Gamma$. From this  we deduce that $G_1\cap G_2$ is isomorphic to $G\cap \mathrm{AGL}_{f}(p)$. Since $|\Aut(\mathrm{PSL}_2(q))|=f(q^2-1)q$ and $|\Omega|=q+1$, we deduce that $G_1\cap G_2$ has order a divisor of $f(q-1)q$. Therefore $|\mathrm{AGL}_{f}(p)|=q|\mathrm{GL}_{f}(p)|$ divides $2f(q-1)q$ (observe that the extra ``2'' in front of $fq(q-1)$ takes in account the case that $G=\Alt(\Omega)$ and $G\cap\mathrm{AGL}_{f}(p)$ has index $2$ in $\mathrm{AGL}_{f}(p)$). Therefore $|\mathrm{GL}_{f}(p)|$ divides $2f(q-1)$.  The inequality $|\mathrm{GL}_{f}(p)|\le 2f(p^{f}-1)$ is satisfied only when $f=1$ or $p=f=2$. When $p=f=2$, we have $|\Omega|=5$ and hence $G_2=\Alt(\Omega)$, which is not the case. Thus $q=p$ and $f=1$. In particular, ${\bf F}^*(G_2)=\mathrm{PSL}_2(p)$, for some prime number $p$. Now, we obtain part~{\bf (g)} and~{\bf (h)} depending on whether $G=\Sym(\Omega)$ or $G=\Alt(\Omega)$.
\end{proof}

\begin{hypothesis}\label{hyp:54}{\em Let $G$ be either $\Sym(\Omega)$ or $\Alt(\Omega)$, let $\Gamma$ be a subset of $\Omega$ with $1\le |\Gamma|<|\Omega|/2$, let $G_1:=\nor G{\Gamma}$, let $G_2$ and $G_3$ be maximal subgroups of $G$ and let $H:=G_1\cap G_2\cap G_3$. Assume that
$\mathcal{O}_G(H)$ is Boolean of rank $3$ with maximal elements $G_1,G_2$ and $G_3$.}
\end{hypothesis}

\begin{theorem}\label{thrm:AS3}
Assume Hypothesis~$\ref{hyp:54}$. Then, relabeling the indexed set $\{1,2,3\}$ if necessary,
one of the following holds:
\begin{enumerate}
\item\label{enu1111}$G=\Sym(\Omega)$, $G_2$ is an imprimitive group having $\Gamma$ as a block of imprimitivity and $G_3=\Alt(\Omega)$.
\item\label{enu2222}$G=\Sym(\Omega)$, $|\Gamma|=1$, $G_3=\Alt(\Omega)$, $G_2\cong\mathrm{PGL}_2(p)$ for some prime $p$ and $|\Omega|=p+1$.
\item\label{enu4444}$G=\Alt(\Omega)$, $|\Gamma|=1$, $G_2\cong G_3\cong M_{24}$ and $|\Omega|=24$.
\end{enumerate}
\end{theorem}
\begin{proof}
A computation shows that the largest Boolean lattice in $\Alt(\Omega)$ when $|\Omega|=7$ has rank $2$. Hence, in the rest of our argument we suppose that $|\Omega|\ne 7$; in particular, part~\eqref{enu222HELL} in Theorem~\ref{thrm:AS2} does not arise.

We apply Theorem~\ref{thrm:AS2} to the pairs $\{G_1,G_2\}$ and $\{G_1,G_3\}$. Relabeling the indexed set $\{2,3\}$ if necessary, we have to consider in turn one of the following case:
\begin{description}
\item[case A] $G_2$ and $G_3$ are imprimitive (hence $G_2$ and $G_3$ are stabilizers of  non-trivial regular partitions having $\Gamma$ as one block);
\item[case B]$G_2$ is imprimitive and $G_3$ is primitive;
\item[case C]$G_2$ and $G_3$ are primitive.
\end{description}

\smallskip

\noindent\textsc{Case A: }Since $\mathcal{O}_{G}(G_2\cap G_3)$ is Boolean of rank $2$, from Lemma~\ref{l:0}, we deduce that either $G_2\cap G_3$ is transitive or $G_2$ or $G_3$ is the stabilizer of a $(|\Omega|/2,2)$-regular partition. As $|\Gamma|\ne |\Omega|/2$, we deduce that $G_2\cap G_3$ is transitive. Therefore, we are in the position to apply Theorem~\ref{thrm:AS} to the pair $\{G_2,G_3\}$. However, none of the possibilities there can arise here because both $G_2$ and $G_3$ have $\Gamma$ as a block of imprimitivity  and $1\le |\Gamma|<|\Omega|/2$.

\smallskip

\noindent\textsc{Case B: } From Theorem~\ref{thrm:AS2}, we have that $\Gamma$ is a block of imprimitivity for $G_2$. If $G_3=\Alt(\Omega)$, then we obtain~\eqref{enu1111}. Suppose then $G_3\ne \Alt(\Omega)$.  As $|\Gamma|\ne |\Omega|/2$, Lemma~\ref{l:0} implies that $G_2\cap G_3$ is transitive and hence we may apply Theorem~\ref{thrm:AS} to the pair $\{G_2,G_3\}$. In particular, Theorem~\ref{thrm:AS} part~\eqref{enu2} holds and hence $G_3$ is an affine primitive group and $G_2$ is the stabilizer of a $(n/2,2)$-regular partition. Thus $|\Gamma|=|\Omega|/2$, which is a contradiction.

\smallskip

\noindent\textsc{Case C: } Suppose that either $G_2$ or $G_3$ equals $\Alt(\Omega)$. Relabeling the indexed set $\{2,3\}$ if necessary, we may suppose that $G_3=\Alt(\Omega)$. In particular, $G=\Sym(\Omega)$. Now, Theorem~\ref{thrm:AS2} implies that $|\Gamma|=1$, $G_2\cong\mathrm{PGL}_2(p)$ for some prime $p$ and $|\Omega|=p+1$. Therefore, we obtain~\eqref{enu2222}. 

It remains to consider the case that $G_2$ and $G_3$ are both primitive and both different from $\Alt(\Omega)$. As $|\Omega|\ne 7$, Theorem~\ref{thrm:AS2} implies that $|\Gamma|=1$, $G_2$ and $G_3$ are one of the groups described in part~\eqref{enu333}. Now, $G_1\cong \Sym(\Omega\setminus\Gamma)$ or $G_1\cong\Alt(\Omega\setminus\Gamma)$, depending on whether $G=\Sym(\Omega)$ or $G=\Alt(\Omega)$. Moreover, $\mathcal{O}_{G}(G_2\cap G_3)$ is a Boolean lattice of rank $2$ having $G_2$ and $G_3$ as maximal elements. From Lemma~\ref{l:1}, we deduce that either $G_2\cap G_3$ acts primitively on $\Omega$, or $G=\Alt(\Omega)$, $G_2\cap G_3=\nor G \Sigma$ for some $(2,4)$-regular partition $\Sigma$. In the latter case, we see with a computation that the lattice $\mathcal{O}_G(G_1\cap G_2\cap G_3)$ is not Boolean (see also Figure~\ref{fig1}). Therefore 
\begin{align*}G_2\cap G_3 \textrm{  acts primitively on }\Omega.
\end{align*} 

Consider then $H:=G_1\cap G_2\cap G_3$ and suppose that $H$ is intransitive on $\Omega\setminus \Gamma$. Since $|\Omega\setminus\Gamma|=|\Omega|-1$, $H$ has an orbit $\Delta\subseteq \Omega\setminus\Gamma$ with $1\le |\Delta|<|\Omega|/2$. Then $\nor G \Delta\in\mathcal{O}_G(H)$ and $\nor G\Delta$ is a maximal element of $\mathcal{O}_G(H)$, contradicting the fact that $G_1$ is the only intransitive element in $\mathcal{O}_G(H)$. Thus $H$ is transitive on $\Omega\setminus\Gamma$. Therefore
\begin{align}\label{eq:1}G_2\cap G_3 \textrm{  acts 2-transitively on }\Omega.
\end{align}

Suppose that $G_2$ is as in Theorem~\ref{thrm:AS2}~\eqref{enu333}~{\bf (a)}, that is, $G_2\cong \mathrm{AGL}_d(2)$ for some $d\ge 3$.  Let $V_2$ be the socle of $G_2$. From Lemma~\ref{l:orcoboiaboia} applied with applied with $H$ there replaced by $G_2\cap G_3$ here , we have either $V_2\le G_2\cap G_3$ or $|\Omega|=8$, $G=\Alt(\Omega)$ and $G_2\cap G_3\cong\mathrm{PSL}_2(7)$. In the second case, $G_1\cap G_2\cap G_3\cong C_7\rtimes C_3$; however, a computation yields that $\mathcal{O}_{\Alt(8)}(C_7\rtimes C_3)$ is not Boolean of rank $3$. Therefore, $V_2\le G_2\cap G_3$. The only primitive groups in Theorem~\ref{thrm:AS2}~\eqref{enu333} with $|\Omega|$ a power of a prime are $\mathrm{AGL}_d(2)$ or $\mathrm{PSL}_2(p)$ when $p+1=2^d$. In particular, either $G_3\cong \mathrm{AGL}_d(2)$, or $G_3\cong \mathrm{PSL}_2(p)$ and $p+1=2^d$. In the second case, since the elementary abelian $2$-group $V_2$ is contained in $G_2\cap G_3$, we deduce that $\mathrm{PSL}_2(p)$ contains an elementary abelian $2$-group of order $2^d$, which is impossible. Therefore, $G_3\cong \mathrm{AGL}_d(2)$. Let $V_3$ be the socle of $G_3$. From Lemma~\ref{l:orcoboiaboia}, we deduce $V_3\le G_2\cap G_3$. In particular, $V_2\unlhd G_2\cap G_3$ and $V_3\unlhd G_2\cap G_3$. Since $G_2\cap G_3$ is primitive, we infer $V_2=V_3$ and hence $G_2=\nor G{V_2}=\nor G{V_3}=G_3$, which is a contradiction. 

Suppose that $G_2$ is as in Theorem~\ref{thrm:AS2}~\eqref{enu333}~{\bf (b)}, that is, $G_2\cong \mathrm{Sp}_{2m}(2)$. To deal with both actions simultaneously we set $\Omega^+:=\Omega$ when $|\Omega|=2^{m-1}(2^m+1)$ and $\Omega^-:=\Omega$ when $|\Omega|=2^{m-1}(2^m-1)$. We can read off from~\cite[Table~1]{LPS}, the maximal subgroups of $G_2$ transitive on either $\Omega^+$ or $\Omega^-$ (this is our putative $G_2\cap G_3$). Comparing these candidates with the list of $2$-transitive groups, we see that none of these groups is $2$-transitive, contradicting~\eqref{eq:1}.

Suppose that $G_2$ is as in Theorem~\ref{thrm:AS2}~\eqref{enu333}~{\bf (c)}, that is, $G_2\cong HS$. The only maximal subgroup of $G_2$ primitive on $\Omega$ is $M_{22}$ in its degree $176$ action. Thus $G_2\cap G_3\cong M_{22}$ in its degree $176$ action. However, this action is not $2$-transitive, contradicting~\eqref{eq:1}.

Suppose that $G_2$ is as in Theorem~\ref{thrm:AS2}~\eqref{enu333}~{\bf (d)}, that is, $G_2\cong Co_3$. From~\cite[Table~$6$]{LPS}, we see that $Co_3$ has no proper subgroup acting primitively on $\Omega$. Therefore this case does not arise in our investigation.

Suppose that $G_2$ is as in Theorem~\ref{thrm:AS2}~\eqref{enu333}~{\bf (e)}, that is, $G_2\cong M_{12}$. In particular, $G_1\cap G_2\cong M_{11}$. We have computed the maximal subgroups of $M_{11}$ with the help of a computer, up to conjugacy, we have five maximal subgroups of $M_{11}$: one of them is our putative $G_1\cap G_2\cap G_3$. For each of these five subgroups, we have computed the orbits on $\Omega$. Observe that one of this orbit is $\Gamma$. If $G_1\cap G_2\cap G_3$ was intransitive on $\Omega\setminus\Gamma$, then $\mathcal{O}_G(H)$ contains a maximal intransitive subgroup which is not $G_1$, contradicting our assumptions. Among the five choices, there is only one (isomorphic to $\mathrm{PSL}_2(11)$) which is transitive on $\Omega\setminus\Gamma$. Thus $G_1\cap G_2\cap G_3\cong \mathrm{PSL}_2(11)$. Next, we have computed $\mathcal{O}_{\Alt(12)}(\mathrm{PSL}_2(11))$ and we have checked that it is not Boolean (it is a lattice of size $6$).

Suppose that $G_2$ is as in Theorem~\ref{thrm:AS2}~\eqref{enu333}~{\bf (f)}, that is, $G_2\cong M_{24}$. The only maximal subgroup of $M_{24}$ acting primitively is $\mathrm{PSL}_2(23)$. Thus $G_2\cap G_3\cong \mathrm{PSL}_2(23)$, and $G_1\cap G_2\cap G_3\cong C_{23}\rtimes C_{11}$.  Now, $\mathcal{O}_{G_1}(G_1\cap G_2\cap G_3)\cong \mathcal{O}_{\Alt(23)}(C_{23}\rtimes C_{11})$. Since $\mathcal{O}_{G_1}(G_1\cap G_2\cap G_3)$ is Boolean of rank $2$, so is $\mathcal{O}_{\Alt(23)}(C_{23}\rtimes C_{11})$. We have checked with the help with a computer that $\mathcal{O}_{\Alt(24)}(C_{23}\rtimes C_{11})$ is Boolean of rank $3$ and this gives rise to the marvellous example in~\eqref{enu4444}.

Using the subgroup structure of $\mathrm{PSL}_2(p)$ and $\mathrm{PGL}_2(p)$ with $p$ prime, we see that $\mathrm{PSL}_2(p)$ does not contain a proper subgroup acting primitively on the $p+1$ points of the projective line, whereas the only proper primitive subgroup of $\mathrm{PGL}_2(p)$ acting primitively on the projective line is $\mathrm{PSL}_2(p)$. Thus part~\eqref{enu333}~{\bf (h)} in Theorem~\ref{thrm:AS2} does not arise and if part~\eqref{enu333}~{\bf (g)} in Theorem~\ref{thrm:AS2} does arise, then $G_2\cap G_3\cong \mathrm{PSL}_2(p)$. However this is impossible because this implies that $G_2\cap G_3\le \Alt(\Omega)$ and hence $\Alt(\Omega)$ must be a maximal element of $\mathcal{O}_G(H)$, but we have dealt with this situation already.
\end{proof}

\begin{corollary}\label{thrm:AS4}
Let $H$ be a subgroup of $G$ and suppose that $\mathcal{O}_G(H)$ is Boolean of rank $\ell\ge 3$ and that $\mathcal{O}_G(H)$ contains a maximal element which is intransitive. Then $\ell=3$; moreover,  relabeling the indexed set $\{1,2,3\}$ if necessary, $G_1=\nor G \Gamma$ for some $\Gamma\subseteq\Omega$ with $1\le |\Gamma|<|\Omega|/2$ and
one of the following holds:
\begin{enumerate}
\item\label{enu11111}$G=\Sym(\Omega)$, $G_2$ is an imprimitive group having $\Gamma$ as a block of imprimitivity and $G_3=\Alt(\Omega)$.
\item\label{enu22222}$G=\Sym(\Omega)$, $|\Gamma|=1$, $G_3=\Alt(\Omega)$, $G_2\cong\mathrm{PGL}_2(p)$ for some prime $p$ and $|\Omega|=p+1$.
\item\label{enu33333}$G=\Alt(\Omega)$, $|\Gamma|=1$, $G_2\cong G_3\cong M_{24}$ and $|\Omega|=24$.
\end{enumerate}
\end{corollary}
\begin{proof}
Let $G_1,G_2,\ldots,G_\ell$ be the maximal elements of $\mathcal{O}_G(H)$. Relabeling the indexed set if necessary, we may suppose that $G_1=\nor G\Gamma$, for some $\Gamma\subseteq \Omega$ with $1\le |\Gamma|<|\Omega|/2$. From Theorem~\ref{thrm:AS3} applied to $\mathcal{O}_{G}(G_1\cap G_2\cap G_3)$, we obtain that $G_1,G_2,G_3$ satisfy one of the cases listed there. We consider these cases in turn. Suppose $G_3=\Alt(\Omega)$ and $G_2$ is an imprimitive group having $\Gamma$ as a block of imprimitivity. If $\ell\ge 4$, then we may apply Theorem~\ref{thrm:AS3} to $\{G_1,G_2,G_4\}$ and we deduce that $G_4=\Alt(\Omega)=G_3$, which is a contradiction. Suppose then $|\Gamma|=1$, $G_3=\Alt(\Omega)$, $G_2\cong\mathrm{PGL}_2(p)$ for some prime $p$ and $|\Omega|=p+1$. If $\ell\ge 4$, then we may apply Theorem~\ref{thrm:AS3} to $\{G_1,G_2,G_4\}$ and we deduce that $G_4=\Alt(\Omega)=G_3$, which is a contradiction.

Finally, suppose that $G=\Alt(\Omega)$, $|\Omega|=24$, $|\Gamma|=1$, $G_2\cong G_2\cong M_{24}$. If $\ell\ge 4$, then we may apply Theorem~\ref{thrm:AS3} to $\{G_1,G_2,G_4\}$ and we deduce that $G_4\cong M_{24}$. In particular, $\mathcal{O}_{G_1}(G_1\cap G_2\cap G_3\cap G_4)$ is a Boolean lattice of rank $3$ having three maximal subgroups $G_1\cap G_2,G_1\cap G_3,G_1\cap G_4$ all isomorphic to $M_{23}$. Arguing as usual $G_1\cap G_2\cap G_3\cap G_4$ acts transitively on $\Omega\setminus \Gamma$. Therefore, $M_{23}$ has a chain $M_{23}>A>B>C$ with $C$ maximal in $B$, $B$ maximal in $A$, $A$ maximal in $M_{23}$, with $C$ transitive. However, there is no such a chain.
\end{proof}

\section{Proof of Theorem~$\ref{thrm:main}$}

We use the notation and the terminology in the statement of Theorem~\ref{thrm:main}. If, for some $i\in \{1,\ldots,\ell\}$, $G_i$ is intransitive, then the proof follows from Corollary~\ref{thrm:AS4}. In particular, we may assume that $G_i$ is transitive, for every $i\in \{1,\ldots,\ell\}$. If, for some $i\in \{1,\ldots,\ell\}$, $G_i$ is imprimitive, then the proof follows from Corollary~\ref{cor:111}.  In particular, we may assume that $G_i$ is primitive, for every $i\in \{1,\ldots,\ell\}$. Now, the proof follows from Corollary~\ref{cor:2}.


\section{Large Boolean lattices arising from imprimitive maximal subgroups}\label{sec:construction}

In this section, we prove that $G$ admits Boolean lattices $\mathcal{O}_G(H)$ of arbitrarily large rank, arising from Theorem~\ref{thrm:main} part~\eqref{main1}. Let $\ell$ be a positive integer with $\ell\ge 2$ and let $\Sigma_1,\ldots,\Sigma_\ell$ be a family of non-trivial regular partitions of $\Omega$ with
$$\Sigma_1<\Sigma_2<\cdots<\Sigma_\ell.$$
For each $i\in \{1,\ldots,\ell\}$, we let $$M_i:=\nor G {\Sigma_i}=\{g\in G\mid X^g\in \Sigma_i,\forall X\in \Sigma_i\}$$
be the stabiliser of the partition $\Sigma_i$ in $G$. More generally, for every $I\subseteq\{1,\ldots,\ell\}$, we let
$$M_I:=\bigcap_{i\in I}M_i.$$
When $I=\{i\}$, we have $M_{\{i\}}=M_i$. Moreover, when  $I=\emptyset$, we are implicitly setting $G=M_\emptyset$. We let $H:=M_{\{1,\ldots,\ell\}}$.

 Here we show that, except when $|\Omega|=8$ and $G=\Alt(\Omega)$,
\begin{equation}\label{eq:eqeqeqeqeq}\mathcal{O}_G(H)=\{M_I\mid I\subseteq \{1,\ldots,\ell\}\}
\end{equation}
and hence $\mathcal{O}_G(H)$ is isomorphic to the Boolean lattice of rank $\ell$. As usual, the case $|\Omega|=8$ and $G=\Alt(\Omega)$ is exceptional because of Figure~\ref{fig1}. To prove~\eqref{eq:eqeqeqeqeq}, it suffices to show that, if $M\in\mathcal{O}_G(H)$, then there exists $I\subseteq \{1,\ldots,\ell\}$ with $M=M_I$. 

We start by describing the structure of the groups $M_I$, for each $I\subseteq \{1,\ldots,\ell\}$. Let $i\in \{1,\ldots,\ell\}$. Since $M_i$ is the stabilizer of a non-trivial regular partition $\Sigma_i$, we have
$$M_i\cong G\cap (\Sym(n/n_i)\mathrm{wr}\Sym(n_i)),$$
where $\Sigma_i$ is a $(n/n_i,n_i)$-regular partition. Since $\{\Sigma_i\}_{i=1}^\ell$ forms a chain, we deduce that $n_{i}$ divides $n_{i+1}$, for each $i\in \{1,\ldots,\ell-1\}$. Now, let $i,j\in \{1,\ldots,\ell\}$ with $i<j$. The group $M_{\{i,j\}}=\nor G{\Sigma_i}\cap \nor G{\Sigma_j}$ is the stabiliser in $G$ of $\Sigma_i$ and $\Sigma_j$. Since $\Sigma_i<\Sigma_j$, we deduce that
$$M_{\{i,j\}}\cong G\cap\left(\Sym(n/n_j)\mathrm{wr} \Sym(n_j/n_i) \mathrm{wr} \Sym(n_i)\right).$$
The structure of an arbitrary element $M_I$ is analogous. Let $I=\{i_1,\ldots,i_\kappa\}$ be a subset of $I$ with $i_1<i_2<\ldots <i_\kappa$. Since $\Sigma_{i_1}<\Sigma_{i_{2}}<\cdots <\Sigma_{\kappa}$, we deduce that
$$M_{I}\cong G\cap\left(\Sym(n/n_{i_\kappa})\mathrm{wr} \Sym(n_{i_{\kappa}}/n_{i_{\kappa-1}}) \mathrm{wr}\cdots \mathrm{wr}
\Sym(n_{i_{2}}/n_{i_1})\mathrm{wr} \Sym(n_{i_1})\right).$$
In particular,
$$H\cong G\cap\left(\Sym(n/n_{\ell})\mathrm{wr} \Sym(n_{\ell}/n_{\ell-1}) \mathrm{wr}\cdots \mathrm{wr}\Sym(n_{2}/n_{1})\mathrm{wr} \Sym(n_{1})\right).$$

\smallskip

We show~\eqref{eq:eqeqeqeqeq} arguing by induction on $\ell$. When $\ell=1$, $H=M_1=\nor G{\Sigma_1}$ and $\mathcal{O}_G(H)=\{H,G\}$ because $H$ is a maximal subgroup of $G$ by Fact~\ref{fact2} (recall that we are excluding the case $G=\Alt(\Omega)$ and $|\Omega|=8$ in the discussion here). For the rest of the proof, we suppose $\ell\ge 2$.

\smallskip

Let $M\in\mathcal{O}_G(H)$. Suppose $M$ is primitive. As $H\le M$, we deduce that $M$ contains a $2$-cycle or a $3$-cycle (when $G=\Sym(\Omega)$ or when $n/n_1\ge 3$), or a product of two transpositions (when $G=\Alt(\Omega)$ and $n/n_1=2$). From~\cite[Theorem~$3.3$D and Example~$3.3.1$]{DM}, either $\Alt(\Omega)\le M$ or $|\Omega|\le 8$. In the first case, $M=M_\emptyset$. When $|\Omega|\in \{6,8\}$, we see with a direct inspection that no exception arises (recall that we are excluding the case $G=\Alt(\Omega)$ and $|\Omega|=8$ in the discussion here). Therefore,  $M$ is not primitive. 

Since $M$ is imprimitive, $H\le M$ and $\Sigma_1,\ldots,\Sigma_\ell$ are the only systems of imprimitivity left invariant by $H$, we deduce that $M$ leaves invariant one of these systems of imprimitivity. Let $i\in\{1,\ldots,\ell\}$ be maximum such that $M$ leaves invariant $\Sigma_i$, that is, $M\le M_i$. Fix $X\in \Sigma_{i}$ and consider $\nor M{X}=\{g\in M\mid X^g=X\}$. Consider also the natural projection $$\pi:\nor {M_{i}}X\to \Sym(X)\cong\Sym(n/n_{i}).$$This projection is  surjective. For each $j\in \{1,\ldots,\ell\}$ with $i < j$ consider $\Sigma_{j}':=\{Y\in\Sigma_j\mid Y\subseteq X\}$. By construction $\Sigma_j'$ is a non-trivial regular partition of $X$ and $$\Sigma_{i+1}'< \Sigma_{i+2}'< \cdots <\Sigma_{\ell}'.$$ Moreover,
$$\pi(\nor {M_j}X)=\nor {\Sym(X)}{\Sigma_{j}'}.$$
In particular, as $\cap_{j=i+1}^\ell\nor {\Sym(X)}{\Sigma_j'}=\pi(\nor H X)\le \pi(\nor M X)$, by induction on $\ell$, $$\pi(\nor M X)=\bigcap_{j\in I'}\nor {\Sym(X)}{\Sigma_j'},$$
for some $I'\subseteq\{i+1,\ldots,\ell\}$. Now, if $I'\ne \emptyset$, then the action of $\nor MX$ on $X$ leaves invariant some $\Sigma_{j}'$, for some $j\in I'$. Since $\Sigma_i< \Sigma_j$ and since $M$ leaves invariant $\Sigma_i$, it is not hard to see that $M$ leaves invariant $\Sigma_j$. However, as $i<j$, we contradict the maximality of $i$. Therefore $I'=\emptyset$ and hence $$\pi(\nor M X)=\Sym(X).$$

Let $H_{(\Omega\setminus X)}$ and $M_{(\Omega\setminus X)}$ be the point-wise stabilizer of $\Omega\setminus X$ in $H$ and in $M$, respectively. Thus $H_{(\Omega\setminus X)}\le M_{(\Omega\setminus X)}\le \Sym(X)$. From the definition of $H$ and from the fact that $X$ is a block of $\Sigma_i$, we deduce 
\[
H_{(\Omega\setminus X)}\cong
\begin{cases}
 \Sym(n/n_\ell)\mathrm{wr}\Sym(n_\ell/n_{\ell-1})\mathrm{wr}\cdots \mathrm{wr}\Sym(n_{i+1}/n_{i})&\textrm{when }G=\Sym(\Omega),\\
\Alt(n/n_i)\cap (\Sym(n/n_\ell)\mathrm{wr}\Sym(n_\ell/n_{\ell-1})\mathrm{wr}\cdots \mathrm{wr}\Sym(n_{i+1}/n_{i}))&\textrm{when }G=\Alt(\Omega).\\
\end{cases}
\] 
We claim that 
\begin{equation}\label{eqbasta!}\Alt(X)\le M_{(\Omega\setminus X)}.
\end{equation} When $i=\ell$, this is clear because in this case $\Alt(X)\le H_{(\Omega\setminus X)}$ from the structure of $H_{(\Omega\setminus X)}$. Suppose then $i\le \ell-1$. Assume first that either $n/n_\ell\ge 3$ or $n/n_i=|X|\ge 5$.
From the description of $H_{(\Omega\setminus X)}$ and from $i\le \ell-1$, it is clear that $H_{(\Omega\setminus X)}$ contains a permutation $g$ which is either a cycle of length $3$ or the product of two transpositions. Define $V:=\langle g^m\mid m\in \nor M X\rangle$. As $H\le M$, we deduce that $g\in M_{(\Omega\setminus X)}$ and hence $V\le M_{(\Omega\setminus X)}$. Since $\pi(\nor M X)=\Sym(X)$, we get $V\unlhd \Sym(X)$ and hence $V=\Alt(X)$. In particular, our claim is proved in this case. It remains to consider the case that $n/n_\ell=2$ and $|X|<5$. As $i\le \ell-1$, this yields $i=\ell-1$, $n/n_\ell=n_\ell/n_{\ell-1}=2$ and $|X|=4$. Observe that in this case, the group $V$ has order $4$ and is the Klein subgroup of $\Alt(X)$. When $G=\Sym(\Omega)$, $H_{(\Omega\setminus X)}$ contains a transposition and hence we may repeat this argument replacing $g$ with this transposition; in this case, we deduce $M_{(\Omega\setminus X)}=\Sym(X)$ and hence our claim is proved also in this case. Assume then $G=\Alt(\Omega)$, $i=\ell-1$, $n/n_\ell=n_\ell/n_{\ell-1}=2$ and $|X|=4$. Among all elements $h\in \nor M X$ with $\pi(h)$ a cycle of length $3$, choose $h$ with the maximum number of fixed points on $\Omega$. Assume that $h$ fixes point-wise $X'$, for some $X'\in \Sigma_i$. From the structure of $H$, we see that $H$ contains a permutation $g$ normalizing both $X$ and $X'$, acting on both sets as a transposition and fixing point-wise $\Omega\setminus (X\cup X')$. Now, a computation shows that $g^{-1}h^{-1}gh$ acts as a cycle of length $3$ on $X$ and fixes point-wise $\Omega\setminus X$, that is, $g^{-1}h^{-1}gh\in M_{(\Omega\setminus X)}$. In particular, $\Alt(X)\le M_{(\Omega\setminus X)}$ in this case. Therefore, we may suppose that $h$ fixes point-wise no block $X'\in\Sigma_i$. Assume that $h$ acts as a cycle of length $3$ on three blocks $X_1,X_2,X_3\in\Sigma_{i}$, that is, $X_1^h=X_2$, $X_2^h=X_3$ and $X_3^h=X_1$. From the structure of $H$, we see that $H$ contains a permutation $g$ normalizing both $X$ and $X_1$, acting on both sets as a transposition and fixing point-wise $\Omega\setminus(X\cup X_1)$. Now, a computation shows that $g^{-1}h^{-1}gh$ acts as a cycle of length $3$ on $X$, as a transposition on $X_1$, as a transposition on $X_2$ and fixes point-wise $\Omega\setminus (X\cup X_1\cup X_2)$. In particular, $(g^{-1}h^{-1}gh)^2$ acts as a cycle of length $3$ and fixes point-wise $\Omega \setminus X$. Thus $(g^{-1}h^{-1}gh)^2\in M_{(\Omega\setminus X)}$ and $\Alt(X)\le M_{(\Omega\setminus X)}$ also in this case. Finally, suppose that $h$ fixes set-wise but not point-wise each block in $\Sigma_i$. In particular, for each $X'\in\nor MX$, we have $X'^h=X'$ and $h$ acts as a cycle of length $3$ on $X'$. Let $X'\in \Sigma_i$ with $X'\ne X$. From the structure of $H$, we see that $H$ contains a permutation $g$ normalizing both $X$ and $X'$, acting on both sets as a transposition and fixing point-wise $\Omega\setminus (X\cup X')$. Now, a computation shows that $g^{-1}h^{-1}gh$ acts as a cycle of length $3$ on $X$ and on $X'$ and fixes point-wise $\Omega\setminus (X\cup X')$. As $h$ was choosen with the maximum number of fixed points with $\pi(h)$ having order $3$, we deduce that $\Omega=X\cup X'$, that is, $n=8$. In particular, we end up with the exceptional case in Figure~\ref{fig1}, which we are excluding in our discussion. Therefore,~\eqref{eqbasta!} is now proved.

Let $K_i$ be the kernel of the action of $M_i$ on $\Sigma_i$. Thus $$K_i=G\cap\prod_{X\in\Sigma_i}\Sym(X)\cong G\cap \Sym(n/n_i)^{n_i}.$$ From~\eqref{eqbasta!}, we deduce $$\Alt(n/n_i)^{n_i}\cong \prod_{X\in \Sigma_i}\Alt(X)\le M.$$ 
As $H\le M$, we obtain $K_i=H(\prod_{X\in \Sigma_i}\Alt(X))\le M$.
%
%
 
\smallskip

Since $\Sigma_1< \Sigma_{2}< \cdots < \Sigma_i$, for every $j\in \{1,\ldots,i\}$, we may consider $\Sigma_j$ as a regular partition of $\Sigma_i$. More formally, define $\Omega'':=\Sigma_i$ and define $\Sigma_j'':=\{\{Y\in \Sigma_i\mid Y\subseteq Z\}\mid Z\in\Sigma_j\}$. Thus $\Sigma_j''$ is the quotient partition of $\Sigma_j$ via $\Sigma_i$.
  Clearly, $M_j/K_i=\nor {M_i}{\Sigma_j''}$.  Applying our  induction hypothesis to the chain $\Sigma_{1}''<\cdots <\Sigma_{i}''$, we have $M/K_i=M_I/K_i$, for some subset $I$ of $\{1,\ldots,i\}$. Since $K_i\le M$, we deduce $M=M_I$.

\section{Large Boolean lattices arising from primitive maximal subgroups}\label{sec:construction2}

\begin{lemma}\label{proprioquelle}Let $\Sigma$ be a $(c,d)$-regular partition of $\Omega$. Given a transitive subgroup $U$ of $\Sym(d)$, we identify the group $X=\Sym(c) \mathrm{wr }\, U$ with a subgroup of $\nor {\Sym(\Omega)}{\Sigma}$.
	If $X$ normalizes a regular partition $\tilde \Sigma$ of $\Omega$, then $\tilde \Sigma\le \Sigma.$ 
\end{lemma}
\begin{proof}
Let $A$ and $\tilde A$ be blocks, respectively, 	of $\Sigma$ and $\tilde \Sigma$ with $A\cap \tilde A \neq\emptyset$ and  let $a\in A\cap \tilde A.$  Then, for every $z\in A\setminus \{a\},$ the transposition $(a,z)\in X$ fixes at least one element of $\tilde A$ and therefore $(a,z)$ normalizes $\tilde A$ and consequently $z\in \tilde A.$ Therefore, either $A\subseteq \tilde A$ or $\tilde A\subseteq A$. From this, it follows that either $ \Sigma \leq \tilde \Sigma$ 
or $\tilde \Sigma \leq \Sigma.$ We can exclude the first possibility, because $\nor X A$ acts on $A$ as the symmetric group $\Sym(A).$
\end{proof}
	
Since we aim to prove that there exist Boolean lattices of arbitrarily large rank of the type described in Thereom~\ref{thrm:main}~\eqref{main3}, we suppose  $n=|\Omega|$ is odd. Let $\ell$ be an integer with $\ell\ge 3$ and let 
$$\mathcal F_1<\cdots < \mathcal  F_{\ell}$$
be a chain of regular product structures on $\Omega$. In particular, $\mathcal{F}_{\ell}$ is a regular $(a,b)$-product structure for some integers $a\ge 5$ and $b\ge 2$ with $a$ odd and $n=a^b$. From the partial order  in the poset of regular product structures, we deduce that we may write $b=b_1\cdots b_\ell$ such that, if we set $d_i:=b_i\cdots b_\ell $ and $c_i:=b/d_i$, then  $\mathcal F_{\ell+1-i}$ is a regular $(a^{c_i},d_i)$-product structure, for every $i\in \{1,\ldots,\ell\}$. 

Let $M_i:=\nor {\Sym(\Omega)}{\mathcal F_i} \cong \Sym(a^{c_i})\mathrm{wr} \Sym(d_i)$ and let $H:=M_1\cap \dots \cap M_{\ell}.$ We have
$$H:=\Sym(a) \mathrm{wr} \Sym(b_1) \mathrm{wr} \Sym(b_2) \mathrm{wr} \cdots \mathrm{wr} \Sym(b_\ell)$$ as a permutation group of degree $n$. 
Moreover, if $I$ is a  subset of $\{1,\ldots,\ell\}$, we let $M_I:=\cap_{i\in I}M_i$, where we are implicitly setting $M_\emptyset=\Sym(n)$. In particular, if $I=\{r_1,\ldots,r_s\}$, then 
 $M_{I}$  is isomorphic to
$$\Sym(a^{b_1\cdots b_{r_1-1}})\mathrm{wr} \Sym(b_{r_1}\cdots b_{r_2-1})\mathrm{wr} \cdots \mathrm{wr}  \Sym(b_{r_s}\cdots b_{\ell})
.$$

For proving that $\mathcal{O}_G(H)$ is Boolean of rank $\ell$, we need to show that, for every $K\in\mathcal{O}_G(H)$, there exists $I\subseteq\{1,\ldots,\ell\}$  with $K=M_I$.

We may identity $H$ with the wreath product $\Sym(a)\mathrm{wr}\, X$ with $X=\Sym(b_1) \mathrm{wr}\, \Sym(b_2) \mathrm{wr}\, \cdots \mathrm{wr}\, \Sym(b_\ell)$, where $X$ has degree $b$ and is endowed of  the imprimitive action of the itereted wreath product and $\Sym(a) \mathrm{wr}\, X$ is primitive of degree $n=a^b$ and  is endowed of the primitive action of the wreath product.

\begin{lemma}\label{soloquelle}
If $H$ normalizes a regular product structure $\mathcal F,$ then $\mathcal F\in \{\mathcal F_1,\ldots,\mathcal F_\ell\}$.
\end{lemma}

\begin{proof}The group  $H=\Sym(a)\mathrm{wr}\, X$ is semisimple and not almost simple. Since the components of $H$  are isomorphic to $\Alt(a)$ and $a$ is odd, according with the definition in~\cite[Section~$2$]{2}, $H$ is product indecomposable.
	From~\cite[ Proposition 5.9 (5)]{2}, we deduce 
	$\bm{\mathcal{F}}(H)$ is isomorphic to the dual of $\mathcal{O}_H(J)\setminus \{H\}$, 
	where $J:=\nor H L $ is the normalizer of a component $L$ of $H$.
	Since ${\bf F}^*(H)=(\Alt(a))^b$, we have $$J=\Sym(a)\times(\Sym(a)\mathrm{wr}\, Y) =\Sym(a)\times(\Sym(a)^{b-1}\rtimes Y),$$ with $Y$ the stabilizer of a point in the imprimitive action of $X$ of degree $b.$	In particular  $\mathcal{O}_H(J)\setminus\{H\}\cong \mathcal{O}_X(Y)\setminus \{X\}.$

	The proper subgroups of $X$ containing the point-stabilizer $Y$ are in one-to-one  correspondence with the regular partitions $\Sigma$ of $\{1,\dots,b\}$ normalized by $X$ and with at least two blocks. Notice that, for any $i\in\{1,\ldots,\ell\},$ there is an embedding of $X$ in $\Sym(c_i)\mathrm{wr}\, \Sym(d_i),$ and therefore $X$ normalizes a  regular $(c_i,d_i)$-partition, which we call it $\Sigma_{\ell+1-i}.$ A iterated application of Lemma \ref{proprioquelle} implies that $\Sigma_{1}<\dots < \Sigma_{\ell}$ are the unique non-trivial regular partitions normalized by $X.$
	\end{proof}

\begin{theorem}If $H\leq K \leq \Sym(n),$ then $K=M_I$ for some subset 
$I$ of $\{1,\ldots,\ell\}.$
\end{theorem}
\begin{proof}Clearly, without loss of generality we may suppose that $H<K< \Sym(n)$.
We apply \cite[Proposition 7.1]{18} to the inclusion $(H,K)$.  Since $H$ has primitive components isomorphic to $\Alt(a)$, with $a$ odd, only cases (ii,a) and (ii,b) can occur.

Assume that $(H,K)$ is an inclusion of type (ii,a). In this case we  have
$H<K\leq \Sym(a)\mathrm{wr}\, \Sym(b)$. Since $\Sym(a)^b\leq H\leq K$ we deduce that $K=\Sym(a) \mathrm{wr}\, Y;$ with $X\leq Y\leq \Sym(b).$ So it suffices to notice that the only subgroups of $\Sym(b)$ containing $X$ are those of the kind $\Sym(b_1\cdots b_{t_1})\mathrm{wr}\, \Sym(b_{t_1+1}\cdots b_{t_2})\mathrm{wr}\, \cdots \mathrm{wr}\,  \Sym(b_{t_s+1}\cdots b_{\ell}\}$, for some subset $\{t_1,\ldots,t_s\}$ of $\{1,\ldots,\ell\}$. Indeed, this fact follows from Section~\ref{sec:construction}.

Assume that $(H,K)$ is an inclusion of type (ii,b). In this case $n=a^b=\alpha^{\gamma\delta}$, $H$ is a blow-up of a subgroup $Z$ of $\Sym(\alpha^\gamma)$ and $(H,K)$ is a blow up of a natural inclusion $(Z,L)$ where $\Alt(\alpha^\gamma)\leq L\leq \Sym(\alpha^\gamma)$. In particular $H$ normalizes a regular $(\alpha^\gamma,\delta)$-product structure $\mathcal F$. By Lemma~\ref{soloquelle}, we have $\mathcal F\in\{\mathcal F_1,\ldots,\mathcal F_\ell\}$. In particular, $\alpha^\gamma=a^{c_i}$ and $\delta=d_i$ and $Z=\Sym(a) \mathrm{wr}\, \Sym(b_1) \mathrm{wr}\, \Sym(b_2) \mathrm{wr}\, \cdots \mathrm{wr}\, \Sym(b_i).$ Since $a$ is odd, $Z\not\leq \Alt(a^{c_i})$ so $L=\Sym(a^{c_i})$
and  $(\Sym(a^{c_i}))^{d_i}\leq K\leq \Sym(a^{c_i})\mathrm{wr}\, \Sym(d_i)$. If $H$ is maximal in $K$, then $i=1$ and $K=\Sym(a^{b_1})\mathrm{wr}\, \Sym(b_2)\mathrm{wr}\, \cdots \mathrm{wr}\, \Sym(b_\ell)=M_{\{1,\ldots,\ell-1\}}$; otherwise, we can proceed by induction on $\ell$.
\end{proof}


\section{Application to Brown's problem} \label{sec:brown}
In this section we will prove Theorem \ref{thrm:bound} (where \eqref{bound3} is a direct application of Theorem \ref{thrm:main}) which proves the conjecture explained in Introduction and provides a positive answer to the relative Brown's problem in this case. 

\subsection{Some general lemmas}
In this subsection we will prove some lemmas working for every finite group. Let $G$ be a finite group and $H$ a subgroup such that the overgroup lattice $\mathcal{O}_{G}(H)$ is Boolean of rank $\ell$, and let $M_1, \dots, M_{\ell}$ be its coatoms. For any $K$ in $\mathcal{O}_{G}(H)$, let us note $K^{\complement}$ its \emph{lattice-complement}, i.e. $K \wedge K^{\complement} = H$ and $K \vee K^{\complement} = G$. 

\begin{lemma} \label{lem:d.d}
If $\mathcal{O}_{G}(H)$ is Boolean of rank $2$ and if $H$ is normal in $M_i$ ($i=1,2$), then $|M_1:H| \neq |M_2:H|$.
\end{lemma}
\begin{proof}
As an immediate consequence of the assumption, $H$ is normal in $M_1 \vee M_2 = G$, but then $G/H$ is a group and $\mathcal{L}(G/H)$ is Boolean, so distributive, and $G/H$ is cyclic by Ore's theorem, thus $|M_1/H| \neq |M_2/H|$.
\end{proof}

\begin{lemma} \label{lem:2.2}
If $\mathcal{O}_{G}(H)$ is Boolean of rank $2$ then $(|M_1:H|,|M_2:H|) \neq (2,2)$.
\end{lemma}
\begin{proof}
If $(|M_1:H|,|M_2:H|) = (2,2)$ then $H$ is normal in $M_i$ ($i=1,2$), contradiction by Lemma \ref{lem:d.d}.
\end{proof}

\begin{lemma} \label{lem:r1r2}
If $\mathcal{O}_{G}(H)$ is Boolean of rank $\ell \le 2$. Then $\hat{\varphi}(H,G) \ge 2^{\ell-1}$.
\end{lemma}
\begin{proof}
If $\ell=1$ then $$\hat{\varphi}(H,G) = |G:H| - |G:G| \ge 2-1 = 2^{\ell-1}.$$
If $\ell=2$, by Lemma \ref{lem:2.2}, there is $i$ with $|M_i:H| \ge 3$. Then
\begin{align*}
\hat{\varphi}(H,G) 
& = |G:H| - |G:M_1| - |G:M_2| + |G:G| \\
&= |G:H|(1 - |M_1:H|^{-1} - |M_2:H|^{-1}) + 1 \\
&\ge 6(1-1/3-1/2)+1 = 2^{\ell-1}. \qedhere
\end{align*}  
\end{proof}

\begin{remark}[Product Formula] \label{profor}
Let $A$ be a finite group and $B,C$ two subgroups, then  $|B| \cdot |C| = |BC| \cdot |B \cap C|$, so $$|B| \cdot |C| \le |B \vee C| \cdot |B \wedge C| \  \text{ and } \ |B:B \wedge C| \le |B \vee C : C|. $$
\end{remark}

\begin{lemma} \label{lem:ind2sub}
Let $A$ be a finite group and $B,C$ two subgroups. If $|A:C|=2$ and $B \not \subseteq C$ then $|B : B \wedge C | = 2$.
\end{lemma}
\begin{proof}
By Product Formula, $2 \le |B : B \wedge C | \le |A:C| = 2$ because $A = B \vee C$.
\end{proof}

\begin{lemma} \label{decre} 
Let $A$ be an atom of $\mathcal{O}_{G}(H)$. If $K_1, K_2 \in \mathcal{O}_{A^{\complement}}(H)$ with $ K_1 < K_2$, then $$|K_1 \vee A : K_1| \le |K_2 \vee A : K_2|.$$
Equivalently, if $K_1, K_2 \in \mathcal{O}_{G}(A)$ with $ K_1 < K_2$, then $$|K_1 : K_1 \wedge A^{\complement}| \le |K_2 : K_2 \wedge A^{\complement}|.$$
Moreover if $|G:A^{\complement}| = 2$ then  $|K \vee A : K| = 2$, for all $K$ in $\mathcal{O}_{A^{\complement}}(H)$.
\end{lemma}
\begin{proof}
By Product Formula, $$|K_1 \vee A | \cdot |K_2| \le |(K_1 \vee A ) \vee K_2| \cdot |(K_1 \vee A ) \wedge K_2 | $$ but $K_1 \wedge K_2 = K_1$, $K_1 \vee K_2 = K_2$ and $A \wedge K_2 = H$, so by distributivity $$ |K_1 \vee A | \cdot |K_2| \le |K_2 \vee A| \cdot |K_1|. $$
Finally, $A^{\complement} \vee A = G$, so if $H \le K \le A^{\complement}$ and $|G:A^{\complement}| = 2$, then  $$2 \le  |K \vee A : K| \le |A^{\complement} \vee A : A^{\complement}| = 2.$$ It follows that $|K \vee A : K| = 2$.
\end{proof}

\begin{lemma} \label{2<->2}
If $\mathcal{O}_{G}(H)$ is Boolean of rank $2$, then $ |M_1:H| = 2$ if and only if $|G:M_2| = 2$.
\end{lemma}
\begin{proof}  If $|G:M_2| = 2$ then $|M_1:H| = 2$ by Lemma \ref{lem:ind2sub} . Now if $|M_1:H| = 2$ then $H \triangleleft M_1$ and $M_1 = H \sqcup H \tau$ with $\tau H = H \tau$ and $(H \tau)^2 = H$, so $H \tau^2 = H$ and $\tau^2 \in H$. Now $M_2 \in (H,G)$ open, then $\tau M_2 \tau^{-1} \in (\tau H \tau^{-1},\tau G \tau^{-1}) = (H,G)$, so by assumption $\tau M_2 \tau^{-1} \in \{M_1,M_2\}$.  If $\tau M_2 \tau^{-1} = M_1$, then $M_2 = \tau^{-1} M_1 \tau = M_1$, contradiction. So $\tau M_2 \tau^{-1} = M_2$. Now $\tau^2 \in H < M_2$, so $M_2 \tau^2 = M_2$. It follows that $G = \langle M_2, \tau \rangle = M_2 \sqcup M_2 \tau$, and $|G:M_2| = 2$. \end{proof}

\begin{lemma} \label{top2} If there are $K,L \in \mathcal{O}_{G}(H)$ such that $K < L$ and $|L:K|=2$, then there is an atom $A$ such that $L = K \vee A$ and $|G:A^{\complement}| = 2$. 
\end{lemma}
\begin{proof}
By the Boolean structure and because $K$ must be a maximal subgroup of $L$, there is an atom $A$ of $\mathcal{O}_{G}(H)$ such that $L = K \vee A$. Let $$K=K_1 < K_2 < \dots < K_r = A^{\complement} $$ be a maximal chain from $K$ to $A^{\complement}$. Let $L_i = K_i \vee A$, then the overgroup lattice $\mathcal{O}_{L_{i+1}}(K_i)$ is Boolean of rank $2$, now $|L_1:K_1| = 2$, so by Lemma \ref{2<->2}
\begin{align*} & 2=|L_1:K_1| = |L_2:K_2| = \cdots = |L_r:K_r|=|G:A^{\complement}|. \qedhere
\end{align*}  
\end{proof}
Note that for $B$ an index $2$ subgroup of $A$, if $|B|$ is odd then $A = B \rtimes C_2$, but it's not true in general if $|B|$ is even.

\begin{lemma} \label{lem:split}
If there is $i$ such that for all $K$ in $\mathcal{O}_{M_i}(H)$, $|K \vee M_i^{\complement}:K| = |M_i^{\complement}:H|$ then $$\hat{\varphi}(H,G) = (|M_i^{\complement}:H|-1)\hat{\varphi}(H,M_i).$$
\end{lemma}
\begin{proof}
By assumption we deduce that $\hat{\varphi}(H,M_i) = \hat{\varphi}(M_i^{\complement},G)$, but by definition, $\hat{\varphi}(H,G) = |M_i^{\complement}:H|\hat{\varphi}(H,M_i)-\hat{\varphi}(H,M_i)$. The result follows.
\end{proof}

\begin{lemma} \label{lem:split2}
If there is $i$ such that $|M_i^{\complement}:H|=2$ then $\hat{\varphi}(H,G) = \hat{\varphi}(H,M_i)$.
\end{lemma}
\begin{proof}
By assumption and Lemma \ref{top2}, $|G:M_i|=2$, so by Lemma \ref{decre}, if $H \le K \le M_i $ then $|K \vee M_i^{\complement}:K|=2$. Thus, by Lemma \ref{lem:split}, $\hat{\varphi}(H,G) = (2-1)\hat{\varphi}(H,M_i).$ 
\end{proof}

\begin{lemma} \label{lem:2^i}
Let $G$ be a finite group and $H$ a subgroup such that the overgroup lattice $\mathcal{O}_{G}(H)$ is Boolean of rank $\ell$, and let $A_1, \dots, A_{\ell}$ be its atoms. If $|A_i:H| \ge 2^i$ then $\hat{\varphi}(H,G) \ge 2^{\ell-1}$.
\end{lemma}
\begin{proof} 
Let $I$ be a subset of $\{1,\dots, \ell\}$ and let $A_I$ be $\bigvee_{i \in I} A_i$. Then $\mathcal{O}_{G}(H)= \{A_I \ | \ I \subseteq \{1,\dots, \ell\}\}$ and 
$$ \hat{\varphi}(H,G) = \sum_{I \subseteq \{1,\dots, \ell\}} (-1)^{|I|} |G:A_I|. $$
By assumption and Lemma \ref{decre}, if $j \not \in I$ then $|G:A_I| \ge 2^j|G:A_I\vee A_j|$. It follows that $$ |G:A_J| \le \frac{1}{|J|} \sum_{j \in J} 2^{-j} |G:A_{J \setminus \{j\}}|$$ from which we get that
\begin{align*}
\hat{\varphi}(H,G) 
& \ge  \sum_{|I| \text{ even}} |G:A_I| - \sum_{|I| \text{ odd}} \frac{1}{|I|} \sum_{i \in I} 2^{-i} |G:A_{I \setminus \{i\}}| \\
& = \sum_{|I| \text{ even}} |G:A_I| (1-\frac{\sum_{i \not \in I}2^{-i}}{|I|+1}) \\
& = \sum_{|I| \text{ even}}|G:A_I|\frac{|I|+2^{-\ell}+\sum_{i \in I}2^{-i}}{|I|+1} \\
& \ge  |G:A_{\emptyset}|2^{-\ell} = 2^{-\ell} |G:H|  \\
& \ge 2^{-\ell}\prod_{i=1}^{\ell} 2^i = 2^{\ell(\ell-1)/2} \ge 2^{\ell-1}.  \qedhere
\end{align*} 
\end{proof}


\begin{lemma} \label{lem:2^i+}
Let $G$ be a finite group and $H$ a subgroup such that the overgroup lattice $\mathcal{O}_{G}(H)$ is Boolean of rank $\ell$, and let $A_1, \dots, A_{\ell}$ be its atoms. If $|A_i:H| \ge a_i>0$
then $\hat{\varphi}(H,G) \ge (1-\sum_i a_i^{-1})\prod_i a_i$.
\end{lemma}
\begin{proof}
It works exactly as for the proof of Lemma \ref{lem:2^i}.
\end{proof}

\subsection{Proof of Theorem \ref{thrm:bound} \eqref{bound1}}
\begin{proof}
The case $n \le 2$ is precisely Lemma \ref{lem:r1r2}. It remains the case $n=3$. 

If there is $i$ such that $|M_i^{\complement}:H|=2$, then by Lemma \ref{lem:2.2} and the Boolean structure, for all $j \neq i$, $|M_j^{\complement}:H| \ge 3$, and by Lemma \ref{lem:split2}, $\hat{\varphi}(H,G) = \hat{\varphi}(H,M_i)$. But as for the proof of Lemma \ref{lem:r1r2}, we have that $$\hat{\varphi}(H,M_i)\ge 9(1-1/3-1/3)+1 = 2^{n-1}.$$

Else, for all $i$ we have $|M_i^{\complement}:H| \ge 3$. Then (using Lemma \ref{decre})
\begin{align*}
\hat{\varphi}(H,G) 
& = |G:H| - \sum_i|G:M_i^{\complement}| + \sum_i|G:M_i| - |G:G|  \\
&\ge |G:H|(1 - \sum_i|M_i^{\complement}:H|^{-1}) + \sum_i|M_i^{\complement}:H| -1 \\
&\ge 27(1-\sum_i 1/3) + \sum_i (3) - 1 = 8>2^{n-1}. \qedhere
\end{align*} 
\end{proof}

\subsection{Proof of Theorem \ref{thrm:bound} \eqref{bound2'}\eqref{bound2}} Let $M_1, \dots, M_{\ell}$ be the coatoms of $\mathcal{O}_{G}(H)$. 

\noindent The Boolean lattice $\mathcal{O}_{G}(H)$ is called \emph{group-complemented} if $KK^{\complement} = K^{\complement}K$ for every $K \in \mathcal{O}_{G}(H)$.  

\begin{lemma} \label{lem:grpcpted}
If the Boolean lattice $\mathcal{O}_{G}(H)$ is group-complemented then $\hat{\varphi}(H,G) = \prod_i (|G:M_i|-1)$.
\end{lemma}
\begin{proof}
By assumption, $KK^{\complement} = K^{\complement}K$ which means that $KK^{\complement} = K \vee K^{\complement} = G$, which also means (by Product Formula) that $|G:K| = |K^{\complement}:H|$. Then by Lemma \ref{decre}, for all $i$ and for all $K$ in $\mathcal{O}_{G}(M_i^{\complement})$, $|K : K \wedge M_i| = |G:M_i|$. Now for all $K$ in $\mathcal{O}_{G}(H)$ there is $I \subseteq \{1, \dots, \ell\}$ such that $K=M_I=\bigwedge_{i \in I}M_i$, it follows that $|G:K| = \prod_{i \in I} |G:M_i|$ and then 
$$ \hat{\varphi}(H,G) = (-1)^{\ell}\sum_{I \subseteq \{1,\dots, \ell\}} (-1)^{|I|} |G:M_I| = (-1)^{\ell} \sum_{I \subseteq \{1,\dots, \ell\}} \prod_{i \in I} (-|G:M_i|) =   \prod_i (|G:M_i|-1). \qedhere $$  
\end{proof}
Theorem \ref{thrm:bound} \eqref{bound2'} follows from Lemmas \ref{lem:grpcpted} and \ref{lem:2.2}. Moreover, if $G$ is solvable and if $\mathcal{O}_{G}(H)$ is Boolean then it is also group-complemented by \cite[Theorem 1.5]{Luc} and the proof of Lemma \ref{lem:grpcpted}. The proof of Theorem \ref{thrm:bound} \eqref{bound2} follows.

\subsection{Proof of Theorem \ref{thrm:bound} \eqref{bound3}}
\begin{proof}
By Theorem \ref{thrm:bound} \eqref{bound1}, we are reduced to consider $\ell \ge 4$ on the cases (1)-(6) of Theorem \ref{thrm:main} where we take the notations:
\begin{enumerate}
\item Take $G = Sym(\Omega)$. By Section \ref{sec:construction}, the rank $\ell$ Boolean lattice $\mathcal{O}_{G}(H)$ is made of $$M_{I}\cong \Sym(n/n_{i_1})\mathrm{wr} \Sym(n_{i_1}/n_{i_2}) \mathrm{wr}\cdots \mathrm{wr}
\Sym(n_{i_{\kappa-1}}/n_{i_\kappa})\mathrm{wr} \Sym(n_{i_\kappa}),$$
with $I = \{ i_1,i_2,\dots, i_\kappa \} \subseteq \{1,\dots,\ell\}$, but $$ |M_I| = \left(\frac{n}{n_{i_1}}!\right)^{n_{i_1}}\left(\frac{n_{i_1}}{n_{i_2}}!\right)^{n_{i_2}} \cdots \left(\frac{n_{i_{\kappa-1}}}{n_{i_{\kappa}}}!\right)^{n_{i_{\kappa}}}n_{i_{\kappa}}! $$   In particular, with $n_0=n$, $n_{\ell+1} = 1$, $H=M_{\{1,\dots, \ell \}}$ and $A_i = M_i^{\complement}$, we have that $$|H|= \prod_{i=0}^{\ell}\left(\frac{n_{i}}{n_{i+1}}!\right)^{n_{i+1}}, \ \ |A_j| = \left(\frac{n_{j-1}}{n_{j+1}}!\right)^{n_{j+1}} \prod_{i \neq j,j+1 }\left(\frac{n_{i}}{n_{i+1}}!\right)^{n_{i+1}}.$$
It follows that 
$$|A_j:H| = \frac{\left(\frac{n_{j-1}}{n_{j+1}}!\right)^{n_{j+1}}}{\left(\frac{n_{j-1}}{n_{j}}!\right)^{n_{j}}\left(\frac{n_{j}}{n_{j+1}}!\right)^{n_{j+1}}}  = \left[ \frac{\left(\frac{n_{j-1}}{n_{j+1}}!\right)}{\left(\frac{n_{j-1}}{n_{j}}!\right)^{\frac{n_j}{n_{j+1}}}\left(\frac{n_{j}}{n_{j+1}}!\right)} \right]^{n_{j+1}} \ge 3^{n_{j+1}}.$$
Take the atom $B_i := A_{\ell + 1 - i}$ and $m_i := n_{\ell + 1 - i}$, then $$|B_i:H| \ge 3^{m_{i-1}} \ge 3^{2^{i-1}} > 2^i.$$
It follows by Lemma \ref{lem:2^i} that $\hat{\varphi}(H,G) \ge 2^{\ell-1}$. 

Next, if $B_i \subseteq \Alt(\Omega)$ then so is $H$ and obviously $|\Alt(\Omega) \cap B_i : \Alt(\Omega) \cap H| = |B_i : H|$, else by Lemma \ref{lem:ind2sub} $|B_i : \Alt(\Omega) \cap  B_i| = 2$, now $|H : \Alt(\Omega) \cap H| = 1$ or $2$ whether $H \subseteq \Alt(\Omega)$ or not. In any case, $$|\Alt(\Omega) \cap  B_i : \Alt(\Omega) \cap H| \ge |B_i : H|/2 > 3^{2^{i-1}-1},$$ 
and we can also apply Lemma \ref{lem:2^i}.
 
\item Let $A_{\ell} = G_{\ell}^{\complement}$, then $|A_{\ell}:H|=2$. Next, we can order, as above, the remaining atoms $A_1, \dots, A_{\ell-1}$ such that $|A_i:H| \ge 3^{2^{i-1}}$ because by assumption $|A_{\ell} : \Alt(\Omega) \cap A_{\ell}|=2$. The result follows by Lemma \ref{lem:2^i+} because 
$$1-(\frac{1}{2}+\sum_{i=1}^{\ell-1}3^{-2^{i-1}}) \ge \frac{1}{2} - \sum_{i=1}^{\ell-1}3^{-i} = \sum_{i=\ell}^{\infty}3^{-i} =  \frac{3}{2}3^{-\ell}.$$

\item Following the notations of Section \ref{sec:construction2}, for $I=\{r_1,r_2,\dots,r_s\}$ we have that
$$ |M_I| = (a^{b_1 \cdots b_{r_1-1}}!)^{b_{r_1}\cdots b_{\ell}} \prod_{i=1}^s ((b_{r_i}\cdots b_{r_{i+1}-1})!)^{b_{r_{i+1}}\cdots b_{\ell}}.$$
\noindent The atom $A_i=M_i^{\complement}$ is of the form $M_{\{i\}^{\complement}}$, whereas, $H=M_{\{1,\dots, \ell\}}$, then (with $b_0=1$)
$$|H| = (a!)^{b_1\cdots b_{\ell}} \prod_{i=1}^{\ell} (b_i!)^{b_{i+1}\cdots b_{\ell}} \ \text{ and } \  A_j= (a^{b_1^{\delta_{1,j}}}!)^{b_1^{-\delta_{1,j}}\prod_{i}b_i} ((b_{j-1}b_{j})!)^{\delta_{1,j}b_{j+1}\cdots b_{\ell}} \prod_{i \neq j-1,j}(b_i!)^{b_{i+1}\cdots b_{\ell}}.$$
Let $j>1$, it follows that
$$ |A_j : H| =  \left[\frac{(b_{j-1}b_{j})!}{((b_{j-1})!)^{b_j}b_{j}!} \right]^{b_{j+1}\cdots b_{\ell}}  \ \text{ and } \   |A_1:H| = \left[\frac{a^{b_1}!}{(a!)^{b_1}b_1!}\right]^{b_{2}\cdots b_{\ell}}.$$
The rest is similar to (1).
\item Similar to (2). 
\item Here $n=a^b$ is a prime power $p^d$ so that $a=p^{d'}$ with $bd'=d$, $b=b_1 \cdots b_{\ell-1}$ and $G_{\ell} = AGL_d(p)$. We can deduce, by using \cite[Theorem 13~(3)]{1}, that
$$AGL_d(p) \cap (\Sym(a^{b_1\cdots b_{r_1}})\mathrm{wr} \Sym(b_{r_1+1}\cdots b_{r_2})\mathrm{wr} \cdots \mathrm{wr} \Sym(b_{r_s+1}\cdots b_{\ell-1}))$$ $$= AGL_{d'b_1\cdots b_{r_1}}(p)\mathrm{wr} \Sym(b_{r_1+1}\cdots b_{r_2})\mathrm{wr} \cdots \mathrm{wr}  \Sym(b_{r_s+1}\cdots b_{\ell-1}).$$
But $|AGL_{k}(p)| = p^k \prod_{i=0}^{k-1} (p^k - p^i)$. The rest is similar to (3).
\item Similar to (2). \qedhere  
\end{enumerate}
\end{proof}


\thebibliography{10}
\bibitem{1}M.~Aschbacher, Overgroups of primitive groups, \textit{J.
Aust. Math. Soc.} \textbf{87} (2009), 37--82.

\bibitem{2}M.~Aschbacher, Overgroups of primitive groups II,
\textit{J. Algebra} \textbf{322} (2009), 1586--1626.

\bibitem{3}M.~Aschbacher, J.~Shareshian, Restrictions on the structure
of subgroup lattices of finite alternating and symmetric groups,
\textit{J. Algebra} \textbf{322} (2009), 2449--2463.

\bibitem{bpJCTA}M.~Balodi, S.~Palcoux, On Boolean intervals of finite
groups, \textit{J. Comb. Theory, Ser. A}, \textbf{157} (2018), 49--69.

\bibitem{basile}A.~Basile, \textit{Second maximal subgroups of the finite alternating and symmetric groups}, PhD thesis, Australian National
Univ., 2001.

\bibitem{magma} W.~Bosma, J.~Cannon, C.~Playoust, The Magma algebra
system. I. The user language, \textit{J. Symbolic Comput.} \textbf{24}
(3-4) (1997), 235--265.

\bibitem{br} K.~S.~Brown, The coset poset and probabilistic zeta
function of a finite group, \textit{J. Algebra} \textbf{225} (2)
(2000), 989--1012.

\bibitem{peter}P.~J.~Cameron, Finite permutation groups and finite
simple groups, \textit{Bull. Lond. Math. Soc.} \textbf{13} (1981),
1--22.

\bibitem{DM}J.~D.~Dixon, B.~Mortimer, \textit{Permutation Groups},
Graduate Texts in Mathematics \textbf{163}, Springer-Verlag, New York,
1996.

\bibitem{hoffman}R.~M.~Guralnick, Subgroups of prime power index in a simple group, \textit{J. Algebra} \textbf{81} (1983, 304--311.



\bibitem{hal}P.~Hall, The Eulerian functions of a group, \textit{Q. J.
Math., Oxf. Ser.}, \textbf{7} (1936), 134--151.

\bibitem{gareth}G.~A.~Jones, Cyclic regular subgroups of primitive permutation groups, \textit{J. Group Theory} \textbf{5} (2002), 403--407. 

\bibitem{li}C.~H.~Li, The finite primitive permutation groups containing an abelian regular subgroup, \textit{Proc. London Math. Soc. (3)} \textbf{87} (2003), 725--747.

\bibitem{16}M.~W.~Liebeck, C.~E.~Praeger, J.~Saxl, A classification of
the maximal subgroups of the finite alternating and symmetric groups,
\textit{J. Algebra} \textbf{111} (1987), 365--383.

\bibitem{LPSLPS}M.~W.~Liebeck, C.~E.~Praeger, J.~Saxl, On  the
O'Nan-Scott theorem for finite primitive permutation groups,
\textit{J. Australian Math. Soc. (A)} \textbf{44} (1988), 389--396

\bibitem{LPS}M.~W.~Liebeck, C.~E.~Praeger, J.~Saxl, The maximal
factorizations of the finite simple groups and their automorphism
groups, \textit{Mem. Am. Math. Soc.} \textbf{432} (1990)

\bibitem{Luc}A.~Lucchini, Subgroups of solvable groups with non-zero M\" obius function. \textit{J. Group Theory} \textbf{10} (2007), no. 5, 633--639.

\bibitem{or}O.~Ore, Structures and group theory. II, \textit{Duke
Math. J.} \textbf{4} (2) (1938), 247--269.

\bibitem{pPJM}S.~Palcoux, Ore's theorem for cyclic subfactor planar
algebras and beyond, \textit{Pacific J. Math.} \textbf{292} (1)
(2018), 203--221.

\bibitem{pQT}S.~Palcoux, Ore's theorem on subfactor planar algebras,
\textit{To appear in Quantum Topology} (2019), arXiv:1704.00745.

\bibitem{pJA}S.~Palcoux, Dual Ore's theorem on distributive intervals
of finite groups, \textit{J. Algebra} \textbf{505} (2018), 279--287.

\bibitem{pPAMS}S.~Palcoux, Euler totient of subfactor planar algebras,
\textit{Proc. Am. Math. Soc.} \textbf{146} (11) (2018), 4775--4786.

\bibitem{PaPu}P.~P\'{a}lfy, P. Pudl\'{a}k, Congruence lattices of
finite algebras and intervals in subgroup lattices of finite groups,
\textit{Algebra
Universalis} \textbf{11} (1980), 22--27.

\bibitem{18}C.~E.~Praeger, The inclusion problem for finite primitive
permutation groups, \textit{Proc. London Math. Soc. (3)} \textbf{60}
(1990), 68--88.

\bibitem{C3}C.~E.~Praeger, Finite quasiprimitive graphs, in
\textit{Surveys in combinatorics}, London Mathematical Society Lecture
Note Series, vol. 24 (1997), 65--85.

\bibitem{PS}C.~E.~Praeger, C.~Schneider, Permutation groups and
Cartesian decompositions, \textit{Lond. Math. Soc. Lect. Note Ser.}
\textbf{449} (2018).

\bibitem{sw}J.~Shareshian, R.~Woodroofe, Order complexes of coset
posets of finite groups are not contractible, \textit{Adv. Math.}
\textbf{291} (2016), 758--773.

\bibitem{mine}P.~Spiga, Finite primitive groups and edge-transitive
hypergraphs, \textit{J. Algebr. Comb.} \textbf{43} (3) (2016),
715--734.
\bibitem{boobs}J.~Tits, \textit{The Geometric Vein. The Coxeter Festschrift}, Springer-Verlag, New York, 1982.

\bibitem{zi}K.~Zsigmondy, Zur Theorie der Potenzreste,
\textit{Monatsh. Math. Phys.} \textbf{3} (1892), 265--284.
\end{document}